\documentclass[12pt]{amsart}

\usepackage{amsmath,amsthm,amsfonts,amssymb}
\usepackage[mathscr]{euscript}
\usepackage{graphicx}

\newcommand{\M}{\mathscr M}

\newcommand{\N}{\mathbb N}
\newcommand{\Z}{\mathbb Z}
\newcommand{\Q}{\mathbb Q}
\newcommand{\R}{\mathbb R}
\newcommand{\C}{\mathbb C}
\newcommand{\Circle}{\mathbb S^1}
\newcommand{\sm}{\setminus}

\newcommand{\0}{\texttt{0}}
\newcommand{\1}{\texttt{1}}
\newcommand{\s}{\underline{s}}

\newcommand\eps{\varepsilon}

\newtheorem{theorem}{Theorem}[section]
\newtheorem{lemma}[theorem]{Lemma}
\newtheorem{proposition}[theorem]{Proposition}
\newtheorem{definition}[theorem]{Definition}
\newtheorem{corollary}[theorem]{Corollary}

\theoremstyle{definition}
\newtheorem*{remark}{Remark}

\numberwithin{equation}{section}

\renewcommand{\theta}{\vartheta}
\newcommand{\ovl}[1]{\overline{#1}}

\newcommand{\lineclear}{\rule{0pt}{2pt}\newline}

\newcounter{reminder}[section]
\newcounter{treminder}

\newcommand{\reminder}[1]{\stepcounter{reminder}\stepcounter{treminder}$\rhd$\textsl{#1}$\lhd$\marginpar{\thesection.\arabic{reminder}$\lhd\rhd\lhd\rhd\lhd\rhd$}}

\renewcommand{\reminder}[1]{}

\newcommand{\hide}[1]{}

\title[Core Entropy of Quadratic Polynomials]{Core Entropy of Quadratic Polynomials %\\ \rule{0pt}{20pt} --- Draft $\circ$ Comments Invited --- 
}

\author{Dzmitry Dudko}
\author{Dierk Schleicher}
\author[]{With an appendix by Wolf Jung}

\address{G.-A.-Universit\"at zu G\"ottingen, Bunsenstra{\ss}e 3--5, D-37073 G\"ot\-tin\-gen, Germany}
\address{Jacobs University Bremen, Research I, Postfach 750 561, D-28725 Bremen, Germany}
\address{Gesamtschule Aachen-Brand, Rombachstra{\ss}e 99, D-52078 Aachen, Germany}

\begin{document}

\begin{abstract}
We give a combinatorial definition of ``core entropy'' for quadratic polynomials as the growth exponent of the number of certain precritical points in the Julia set (those that separate the $\alpha$ fixed point from its negative). This notion extends known definitions that work in cases when the polynomial is postcritically finite or when the topology of the Julia set has good properties, and it applies to all quadratic polynomials in the Mandelbrot set.

We prove that core entropy is continuous as a function of the complex parameter. In fact, we model the Julia set as an invariant quadratic lamination in the sense of Thurston: this depends on the external angle of a parameter in the boundary of the Mandelbrot set, and one can define core entropy directly from the angle in combinatorial terms. As such, core entropy is continuous as a function of the external angle.

Moreover, we prove a conjecture of Giulio Tiozzo about local and global maxima of core entropy as a function of external angles: local maxima are exactly dyadic angles, and the unique global maximum within any wake occurs at the dyadic angle of lowest denominator. We also describe where local minima occur.

An appendix by Wolf Jung relates different concepts of core entropy and biaccessibility dimension and thus shows that biaccessibility dimension is continuous as well.
\end{abstract}

\maketitle

\section{Introduction and Statement of Results}

Topological entropy of a topological dynamical system $f\colon K\to K$ measures the complexity of the dynamical system, roughly speaking as follows. If $K=\bigcup U_i$ is covered by some number of open sets, denote by $N(n)$ the number of allowed itineraries of length $n$: these are sequences $i(0),\dots,i(n-1)$ for which there exists some $x\in K$ with $f^{\circ k}(x)\in U_{i(k)}$ for $k=0,1,\dots,n-1$. Then the growth exponent of $N(n)$ is the topological entropy of $(K,f)$; more precisely, topological entropy is defined as $h=\limsup_n\frac{1}n \log N(n)$. For a precise definition and equivalent descriptions, see for instance \cite{BrinStuck} or \cite{deMelovanStrien}.

In their seminal paper \cite{MilnorThurston}, Milnor and Thurston investigated topological entropy of continuous interval maps. Specifically for the quadratic family $f_\lambda\colon x\mapsto \lambda x(1-x)$ acting on $I=[0,1]$ they proved that topological entropy as a function of $\lambda$ is monotone and continuous. 

Misiurewicz and Sz{\l}enk \cite{MisiurewiczSzlenk} wrote one of the early papers that gave various equivalent interpretations of topological entropy for interval maps, for example as growth exponents of intervals of monotonicity of $f^{\circ n}$ or of the number of periodic points of period $n$. Since then, there has been a lot of activity on related questions; see for instance \cite{deMelovanStrien} for an overview.

In the last years of his life, William Thurston raised the issue of finding a good definition of topological entropy of a complex polynomial $p$ of degree $d$. Viewing $p$ as a map on $\C$ or on the filled-in Julia set,  the entropy is clearly $\log d$ and not very interesting. The same holds of course when $p$ is a real polynomial (considered as a self-map of $\C$), but in that case there may be a dynamically interesting invariant interval in $\R$ that is invariant and that contains the orbits of the critical points: the restriction to this interval is the dynamically interesting object, and the aforementioned studies were concerned with self-maps of such intervals. What is the analogous object for complex polynomials, or what is an interesting definition of topological entropy?

If the polynomial $p$ is postcritically finite (that is, all critical points are periodic or preperiodic), then it has a natural invariant tree called its Hubbard tree $H$, and Thurston defined the \emph{core entropy} of $p$ as the topological entropy of the restriction to $H$; see \cite{bghkltt,taoli}. This entropy can also be defined in certain other cases: for instance, for certain maps one can define a finite tree in analogy to the Hubbard tree that is still invariant and contains the critical orbit, even if the latter is infinite (in analogy to the case of real polynomials). However, in the general case there are difficulties: the filled-in Julia set may not be path connected so that there may not be a tree, or the critical orbit may be dense in the filled-in Julia set so that any tree would have to be infinite, and topological entropy might seem to be equal to $\log d$ (this happens for a dense set of parameters on $\partial\M$, so if entropy with this definition was continuous then it would have to be constant). Thurston asked for a definition of core entropy that would work in all cases. Moreover, in a seminar in Cornell in the spring of 2012 he raised the question whether core entropy could be continuous as a function of the complex parameter. This resulted in a bet between him and John Hubbard, which was one original inspiration for starting our work.

One possible definition is in terms of \emph{biaccessibility dimension}: that is the Hausdorff dimension of all angles $\theta\in\Circle$ for which there exists another angle $\theta'\neq\theta$ so that the dynamic rays at angles $\theta$ and $\theta'$ land at a common point. Thurston had shown that in the cases for which his definition of core entropy applied that it was equal to biaccessibility dimension (up to a factor of $\log d$). Earlier we had shown \cite{MeerkampSchleicher} that the biaccessibility dimension is always less than $1$, except when the Julia set is an interval, strengthening earlier work by Smirnov \cite{Smirnov}, Zakeri \cite{Zakeri}, and Zdunik \cite{Zdunik}.
Quite recently, Tiozzo wrote an interesting thesis on core entropy and related questions \cite{TiozzoThesis}, in particular with a relation to the thermodynamic formalism. Some additional history, as well as the relation between core entropy and biaccessibility dimension in the general case, are given in the appendix by Wolf Jung.

Bill Thurston inspired a number of people to investigate core entropy. In particular, there are a survey on current work and open problems by Tan Lei~\cite{TanLeiEntropy}, a manuscript by Wolf Jung \cite{Jung}, and two manuscripts by Giulio Tiozzo \cite{TiozzoThesis,TiozzoPaper}.

We propose a general definition of core entropy that coincides with Thurston's in all cases he considered, and that does not require the Julia set to have particular properties (such as pathwise connectivity or the existence of an invariant compact tree that might generalize the Hubbard tree), nor does it require the concepts of thermodynamic formalism or biaccessibility dimension. Instead, our approach is purely combinatorial and thus applies in great generality; here we work only on the case of quadratic polynomials with connected Julia sets. Of course, in the cases where the usual definitions of topological entropy apply, our definition agrees with them.

More precisely, we define two entropy functions:
\begin{itemize}
\item $h\colon \Circle\to [0,\log 2]$, assigning to every external angle $\theta\in\Circle$ the core entropy of the lamination associated to the angle $\theta$.
\item $\tilde h\colon \M\to[0,\log2]$, assigning to every $c\in\M$ (the Mandelbrot set) the core entropy of the quadratic polynomial $z\mapsto z^2+c$. 
\end{itemize}

We describe these definitions in Section~\ref{Sec:Definitions}. Note that we define $\tilde h$ for every $c\in\M$, postcritically finite or not (and could extend it to every parameter $c\in\C$), and similarly for every $\theta\in\Circle$. These definitions are such that if the parameter ray at angle $\theta$ lands (or accumulates) at $c\in\partial \M$, then $h(\theta)=\tilde h(c)$. 

Our first result goes back to the bet between Bill Thurston and John Hubbard in the spring of 2012 and helped settle this bet soon after.

\begin{theorem}[Continuity of Core Entropy] \lineclear
Both entropy functions, $\tilde h\colon\M\to[0,\log 2]$ and $h\colon\Circle\to[0,\log 2]$, are continuous.
\end{theorem}

We prove continuity of $h$ in Theorem~\ref{Thm:Continuity}; the fact that this implies continuity of $\tilde h$ is easy and is explained in Section~\ref{Sec:Definitions}. An independent proof of continuity of core entropy can be found in the recent manuscript \cite{TiozzoPaper}.  

\begin{figure}[tb]
\framebox{
\includegraphics[width=.9\textwidth]{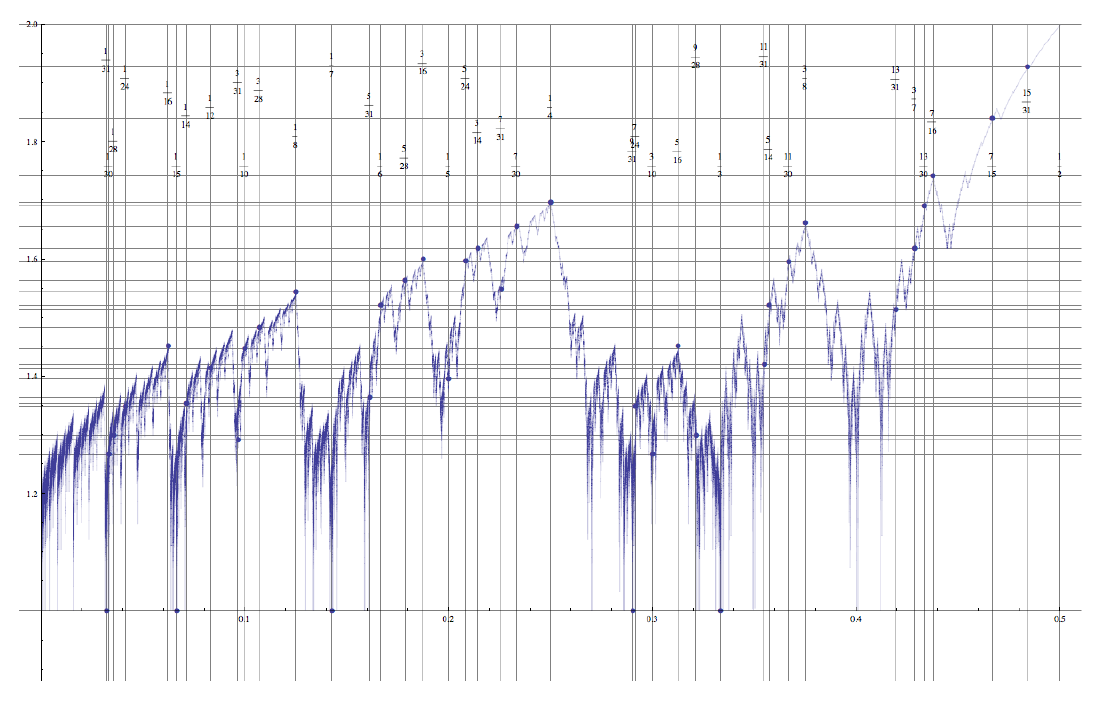}
}
\caption{The graph of the core entropy function $h\colon[0,1/2]\to[0,\log 2]$ (by Bill Thurston).}
\label{Fig:EntropyGraph}
\end{figure}

Our second result settles a conjecture of Giulio Tiozzo \cite[Conjecture~1.6]{TiozzoThesis}.

\begin{theorem}[Local Maximal of Core Entropy]
\lineclear
The entropy function $h\colon\Circle\to[0,\log 2]$ has the following properties.
\begin{enumerate}
\item[a)] 
Every dyadic angle is an isolated local maximum of the entropy function.

\item[b)]
Conversely, every local maximum of $h$ is dyadic. 
\item[c)] 
Within every wake, the entropy function has a unique global maximum, and it occurs at the unique dyadic angle of lowest denominator in the wake.
\end{enumerate}
\end{theorem}
We will prove this in Theorem~\ref{Thm:TiozzoConj}.

In Section~\ref{Sec:Definitions}, we give all relevant definitions, in particular of the two entropy functions $h$ and $\tilde h$, and of invariant laminations; we define a wake as a closed interval in $\Circle$ bounded by two rational angles so that the corresponding rays land together (or more generally as any closed interval for which the two boundary angles have the same angled internal address, so the corresponding parameter rays land together at the same point in $\M$, or at least in its combinatorial models).

The fundamental construction will be carried out in Section~\ref{Sec:TopSurgery}: we define a surgery construction of dyadic Hubbard trees (that is, Hubbard trees associated to external angles $a/2^k$) that lowers the dyadic exponent $k$  and increases the core entropy. We also describe the vein structure of $\M$ and introduce a relation $\lhd$ between dyadic veins.

From this construction, it is no longer difficult to prove Tiozzo's conjecture on maxima of $h$; this will be done in Section~\ref{Sec:IrrationalTiozzoConj}: here we have to extend the estimates on entropy to non-dyadic external angles. We also discuss local minima of $h$ and $\tilde h$ in this section.

In Section~\ref{Sec:Continuity} we then discuss continuity of entropy. The main result is that for any parameter $c\in\M$, if entropy is continuous along the (combinatorial) path from $0$ to that endpoint, then it is continuous in $\M$ at each of these points on the path. The difference is that we \emph{assume} continuity on a path, but continuity (the conclusion) makes a claim about an entire neighborhood in $\C$. It is known, by work of Tiozzo and Jung, that entropy along many paths in $\M$ is continuous, in particular to all dyadic endpoints. It remains to prove that entropy is continuous along paths to irrational endpoints of $\M$, and that we prove in Section~\ref{Sec:IrratEndpoints}. 

Finally, in a brief appendix, Wolf Jung discusses the relations of core entropy with biaccessibility dimension, and concludes that biaccessibility dimension is a continuous function as well.

\emph{Acknowledgements}. We would like to thank Henk Bruin, John Hubbard, Tan Lei, Mikhail Lyubich, John Milnor, Bill Thurston, Giulio Tiozzo, Jean-Christophe Yoccoz and especially Wolf Jung for interesting and helpful discussions. In the spring of 2014, we had the opportunity to give various presentations about this result in Bremen, Moscow, and Stony Brook, and we thank the audiences for their questions and suggestions.
Finally, we would like to thank Cornell University and the ICERM institute in Providence for their hospitality and support in the spring of 2012 where many of our initial discussions were carried out that yielded our first version of the proof.
\goodbreak

%\newpage

\section{Definitions}
\label{Sec:Definitions}  
In this section we introduce our combinatorial definition of topological entropy that applies to all (quadratic) polynomials with connected Julia set, whether or not they are locally connected and whether or not they have (generalized) Hubbard trees, and if so whether the latter are compact. One advantage of our approach is that we work in an entirely combinatorial setting, so we never have to worry about topological issues. We work on invariant (quadratic) laminations as introduced by Thurston \cite{ThurstonLaminations}.

For every angle $\theta\in\Circle\setminus\{0\}$, there is a unique invariant quadratic lamination where the minor leaf either ends at $\theta$ or is the degenerate leaf at $\theta$. This lamination will be called $L_\theta$.

A \emph{precritical leaf of generation $n$} in this lamination will be any leaf on the backwards orbit of one of the two major leaves that takes $n$ generations to map to the minor leaf. Arbitrarily choose one of these two major leaves as \emph{preferred major leaf} (in the special case that the minor leaf is degenerate, there is only one major leaf, which is a diameter, and in this case there is no choice).

In an invariant quadratic lamination, the \emph{$\alpha$ gap} is either the leaf connecting the two angles $1/3$ and $2/3$, or it is the unique gap that is fixed by the dynamics (a finite or infinite polygon).

\begin{definition}[Relevant Precritical Leaf and Entropy Associated to External Angle]
\label{Def:CoreEntropy}
We call a precritical leaf \emph{relevant} if it separates the $\alpha$ gap from its negative and if it is on the backwards orbit of the preferred major leaf, and define $N(n)=N_\theta(n)$ as the number of relevant precritical leaves of generation $n$. We define the \emph{core entropy} of this lamination as $h=h(\theta):=\limsup_n\frac 1 n \log N_\theta(n)$. 

In the special case $\theta=0$, the lamination is trivial and we naturally have $N_{\theta}(n):=0$.
\end{definition}

If the $\alpha$ gap is an infinite polygon, then $N_\theta(n)=0$ except for $n=1$, so $h(\theta)=0$. 

\begin{remark}
A natural question is whether the $\limsup$ in the definition of entropy can be replaced by a simple $\lim$. This is not always so: Wolf Jung kindly pointed out to us the example of $c=c(9/56)$ where the Hubbard tree has the shape of a $\textsf{Y}$ so that the branch point is fixed and one endpoint maps to the second, which maps to the third, which in turn maps to the branch point. Here $N(n)=0$ for infinitely many $n$ while $h=(\log 2)/3>0$. There are thus counterexamples when the dynamics is renormalizable. We prove that immediate satellite renormalizability is the only obstruction (see Corollary~\ref{Cor:ExistenceLimit}). For now, observe that in the postcritically finite non-renormalizable case, the limit exists and equals the $\limsup$ because the associated subshift of finite type is irreducible.

In definition of relevant precritical leaf, we only count those leaves that are preimages of the preferred major leaf. Without this restriction, the count of $N(n)$ would increase by a factor $2$, which would yield the same definition of entropy, but we would lose monotonicity of $N(n)$ as in Lemma~\ref{Lem:Monotonicity} for the stupid reason that the degenerate minor leaf has only half as many preimages (of course, an alternate way to remedy this problem would be to count preimages of a degenerate minor leaf with multiplicity two). 
\end{remark}

Thurston showed that the union of all minor leaves of all invariant quadratic laminations forms itself a lamination, called the \emph{quadratic minor lamination} QML \cite{ThurstonLaminations}. It turns out that $h$ is naturally defined on QML: since both ends of any leaf in QML define the same lamination, we can first extend the definition of $h$ to each separate leaf on QML. Complementary components of leaves in QML are called \emph{gaps}, and they come in two kids: either they are finite polygons (corresponding to Misiurewicz-Thurston parameters) or have infinitely many boundary leaves (and describe hyperbolic components). In both cases, it is easy to see that $h$ is constant on all boundary leaves of any gap, so $h$ naturally extends to the disk on which QML is defined, as a constant function on the closure of each leaf and each gap. 

The equivalence relation defining QML is closed, so the quotient of the supporting closed unit disk by collapsing all leaves to points yields a topological Hausdorff space called the ``abstract Mandelbrot set'' $\M_{abs}$ (this construction is known as Douady's ``pinched disk model'' of $\M$). Since $h$ is constant on fibers of the quotient map $q\colon \ovl{\mathbb D} \to \M_{abs}$, it follows that $h$ is naturally a function on $\M_{abs}$. 

Finally, there is the natural projection $\pi\colon\M\to\M_{abs}$ from the Mandelbrot set $\M$ to the abstract Mandelbrot set $\M_{abs}$. It is defined by mapping every landing point $c(\theta)$ of any rational parameter ray $\theta$ to the equivalence class of the angle $\theta$; then $\pi\colon\partial\M\to\partial\M_{abs}$ is the unique continuous extension, and from here it is easy to extend $\pi$ to a continuous map $\M\to\M_{abs}$ that on each component of the interior is either injective or constant (for details, see Douady~\cite{DouadyCompacts}).

\looseness-1
We thus obtain a unique map $\tilde h=h\circ\pi\colon\M\to[0,\log 2]$, and continuity of $\tilde h$ follows from continuity of $h$, with continuity of $\pi$ being well known.
More specifically, the map $\tilde h$ can also be constructed explicitly as follows.

We define a \emph{ray pair} $RP(\phi,\phi)$ to be the union of two rays (in a dynamical plane or in parameter space) that land (or perhaps accumulate) at a common point, together with the union of their accumlation sets. Every ray pair divides $\C$ into two open components.

\begin{definition}[Entropy Associated to Quadratic Polynomial in $\M$]
\label{Def:CoreEntropyTilde}
Let $p_c(z):=z^2+c$ be any quadratic polynomial with $c\in\C$ for which the critical value is in the Julia set, and let $\theta$ be the external angle of any parameter ray that lands or accumulates at the parameter $c$. The \emph{critical ray pair} will be the ray pair $RP(\theta/2,(1+\theta)/2)$, and a precritical ray pair of generation $n$ will be any ray pair on the backwards orbit of the critical ray pair that takes $n$ iterations to map to the ray $R(\theta)$.

If $p_c$ is such that the critical value is in the Fatou set, then it has an attracting or parabolic orbit and there is a unique periodic characteristic ray pair that lands on the boundary of the Fatou component containing the critical value; precritical ray pairs are then defined as ray pairs on its backward orbit, and their generation is the number of iterations it takes to map to the characteristic ray pair.

We call a precritical ray pair \emph{relevant} if it separates the $\alpha$ fixed point from its negative, and define $\tilde N(n)$ as the number of relevant precritical ray pairs of generation $n$. We define the \emph{core entropy} of $p_c$ as $\tilde h=\tilde h(c):=\limsup_n\frac 1 n \log \tilde N(n)$.
\end{definition}

Note that we use the term ``separation'' in a combinatorial sense: the two rays in a separating ray pairs either land together or accumulate at the same fiber. --- We do not wish to (or need to) deal with topological subtleties such as whether, in the Cremer case, a dynamic ray indeed accumulates at the critical value: every $c\in\M$ has naturally associated one or several external angles $\theta$ that define its dynamics, and this is sufficient for our combinatorial approach.

We need to fix some notation on Hubbard trees of postcritically finite  polynomials. The \emph{Hubbard tree} is a minimal tree within the filled-in Julia set connecting the postcritical set (subject to a natural condition on how to traverse bounded Fatou components in case some critical points are periodic). 
The \emph{marked points} or \emph{vertices} of the Hubbard tree are the endpoints, branch points, and the postcritical points. (In fact, all endpoints are postcritical points; critical points are not included in their own right, but they might be postcritical, for instance when they are periodic, and when the degree is greater than $2$ then they might also be branch points.) Since the set of vertices is forward invariant, every edge (a closed arc connecting two vertices) maps over one or several entire edges, so that the image contains every edge the interior of which intersects the image; in other words, the edges form a Markov partition on the Hubbard tree. --- Here, we will only discuss quadratic polynomials and Hubbard trees.

The usual definition of core entropy is modeled after the postcritically finite case. The finite set of edges on the tree form a Markov chain with associated transition matrix, where the matrix element $M_{i,j}$ is $0$, $1$, or $2$ if the edge $e_i$ covers the edge $e_j$ respectively $0$, $1$, or $2$ times. Having only non-negative real entries, this matrix has a leading eigenvalue which is real, and its logarithm is defined as the core entropy of the given postcritically finite parameter. 
This definition coincides with the classical definition of topological entropy of general dynamical systems, and it also applies in the postcritically infinite case as long as the Hubbard tree is defined (i.e.\ the Julia set is path connected) and still finite. However, the number of edges of the Hubbard trees is not locally bounded even among postcritically finite maps, which makes entropy estimates based on these transition matrices difficult.

It is well known, at least in the postcritically finite case, that if $x$ is any point on the Hubbard tree and $N_x(n)$ is the number of preimages of $x$ of generation $n$, then $h=\limsup_n \log N_x(n)$ is the core entropy. Since each of the finitely many edges, except those within ``renormalizable little Julia sets'', will cover the entire tree after finitely many iterations (Lemma~\ref{Lem:HubbardTreeRenormalization}), one can as well count only those preimages of $x$ that are on an arbitrary subset of the edges of the Hubbard tree, as long as at least one of these edges is not in a renormalizable little Julia set. For instance, instead of counting precritical leaves on $[\alpha,-\alpha]$ (as in our definition above) we may count preimages on $[\alpha,\beta]$ (as we will do in Section~\ref{Sec:Continuity}) or on $[\beta,-\beta]$.

\begin{lemma}[Definitions of Core Entropy Coincide]
\label{Lem:defEntropy} \lineclear
For postcritically finite polynomials, the core entropy as in Definition~\ref{Def:CoreEntropyTilde} coincides with the usual definition (in terms of transition matrices on finite Hubbard trees). 

\looseness-1
If for $p_c$ the critical value is in the impression of the dynamic ray at angle $\theta$,  the core entropy of $p_c$ equals the core entropy of the lamination $L_\theta$. 
\end{lemma}

It is well known that if several parameter rays accumulate at the same parameter in $\M$, then the laminations associated to their corresponding angles coincide, so these angles have the same entropy (if more than two rays accumulate at the same parameter, then the parameter is a Misiurewicz-Thurston parameter and the rays actually land there). 

If $\theta\in\Circle$ is so that the parameter ray of $\M$ at angle $\theta$ lands (in particular, if $\theta$ is rational), we define $c(\theta)$ as the landing point in $\M$ of the parameter ray at angle $\theta$, and within any connected Julia set we define $z(\theta)$ as the landing point of the dynamic ray at angle $\theta$ (if the latter ray lands). 

%\newpage

We define a partial order on $\M$ as follows: if $c,\tilde c$ are two parameters in $\M$, we say that $c\prec \tilde c$ if there is a parameter ray pair $RP(\phi^-,\phi^+)$ at periodic angles that separates $\tilde c$ simultaneously from $c$ and from the origin. 

Following Milnor~\cite{MiOrbits}, a periodic or preperiodic ray pair $RP(\phi^-,\phi^+)$ (with $\phi^-,\phi^+\in\Circle$) in the dynamical plane of $p_c$ is called \emph{characteristic} if the dynamic rays $R(\phi^-)$ and $R(\phi^+)$ land together in such a way that they separate the critical value $c$ from the critical point $0$ as well as from all other rays landing at $\bigcup_{k\ge 0}p_{c}^{\circ k}(x)$.
It is well known that every periodic ray pair $RP(\phi^-,\phi^+)$ has a unique characteristic ray pair on its forward orbit.

  Let us now state the Correspondence Theorem relating the combinatorics of dynamical and parameter ray pairs; see for instance \cite{MiOrbits} or \cite{MandelBranch}. 

\begin{theorem}[Correspondence Theorem]
\label{thm:Corresp} \lineclear
A ray pair $RP(\phi^-,\phi^+)$ in the dynamical plane of $p_c$ with $\phi^\pm$ rational  is characteristic if and only if the parameter rays with angles $\phi^-$ and $\phi^+$ land together and separate the parameters $c$ and $0$ from each other.
\end{theorem}

\begin{lemma}[Monotonicity of Lamination and of Entropy]
\label{Lem:Monotonicity} \lineclear
If $c\prec \tilde c$, then $N_c(n) \le N_{\tilde c}(n)$ for all $n$ and thus $h(c) \le h(\tilde c)$. Moreover, any characteristic leaf in $L_c$ also occurs in $L_{\tilde c}$. 
\end{lemma}
\begin{proof}
It is routine to check that any precritical leaf of $c$ also ``occurs'' in the dynamics (or the lamination) of $\tilde c$: the major leaf (or leaves) in $L_{\tilde c}$ separate the two major leaves in $L_c$. Precritical leaves in $L_{c}$ are preimages of the pair of major leaves (the preimages of this pair are always ``parallel'', i.e.\ not separated by the critical value, because otherwise the forward orbit of the critical value would have to intersect the domain bounded by the two major leaves). Each pair of preimages surrounds one preimage of the major leaf of $\tilde c$ (or a pair of preimages of the major leaves), and when such a preimage separates the $\alpha$ gap from its negative in $L_c$, then it also does so for $L_{\tilde c}$. Therefore, $N_{\tilde c}(n)\ge N_c(n)$ and $\tilde h(\tilde c)\ge \tilde h(c)$. 

The statement about characteristic leaves (or characteristic ray pairs) is well known and follows from the Correspondence Theorem~\ref{thm:Corresp}.
\end{proof}

For the record, we define a \emph{dyadic parameter} as a parameter $c$ that is the landing point of a parameter ray at a dyadic angle $a/2^m$; in this case, we call $m$ the \emph{generation} of $c$.

%\newpage

%\goodbreak

%\newpage

\section{Topological Surgery on dyadic Hubbard trees}
\label{Sec:TopSurgery}

\emph{Combinatorial arcs and veins}.
Conjecturally, the Mandelbrot set is path connected: every $c\in\M$ has an arc $[0,c]$ that connects it to the origin. Such arcs are unique when requiring that they traverse hyperbolic components only along internal rays (radial curves with respect to the parameterization by the multiplier map). In fact, for many parameters $c\in\M$ one can prove that such an arc actually exists (by work of Jeremy Kahn using Yoccoz' puzzle results, and by Johannes Riedl using quasiconformal surgery). However, we only need a combinatorial version of such arcs. One way of defining them is as follows: the abstract Mandelbrot set $\M_{abs}$ (as defined in Section~\ref{Sec:Definitions}) is well known to be locally connected and hence path connected, and it comes with a continuous projection $\pi\colon\M\to\M_{abs}$. For $c\in\M$, there exists a preferred path $\Gamma(c)\subset\M_{abs}$ connecting $\pi(0)$ to $\pi(c)$, and then we define $[0,c]:=\pi^{-1}(\Gamma(c))$. It is not hard to describe this combinatorial arc $[0,c]$ in terms of internal rays of hyperbolic components and of fibers that separate $0$ from $c$; the details are somewhat tedious but not very enlightening.

Specifically if $c$ is a dyadic parameter in $\M$, we call the combinatorial arc $[0,c]$ the \emph{long vein} of $c$. 
The \emph{vein} of $c$ is the shortest closed sub-arc of the long vein connecting $c$ to the union of the long veins of all dyadic parameters of lower generation than $c$.

 If $c$ and $c'$ are two dyadic parameters, then it is well known that their long veins intersect in an arc $[0,x]$, where $x$ is postcritically finite; this result is known as the ``branch theorem'' of $\M$ \cite{Orsay,MandelBranch}. Moreover, in the special case that $c$ and $c'$ are dyadic of equal generation, then $x$ is on the vein of a dyadic parameter $c''$ of lower generation. This means that veins (minus endpoints) are disjoint.

\begin{definition}[Directly Subordinate Parameter $c\lhd c'$]
\label{Def:DirectlySubordinate} \lineclear
We say that $ c$ is \emph{directly subordinate} to $c'$ and write $c\lhd c'$ if the vein of $c$ terminates at an interior point of the vein of $c'$; in addition, any dyadic parameter whose vein terminates at $0$ is declared to be subordinate to $c(1/2)=-2$.
\end{definition}
If $c\lhd c'$, then necessarily the external angle of $c'$ has lower denominator than that of $c$.  
Note that this is not a transitive relation and thus not a partial order. A few directly subordinate dyadic parameters are illustrated in Figure~\ref{Fig:DirectlySubordinate}.

\begin{figure}[htbp]
\framebox{\includegraphics{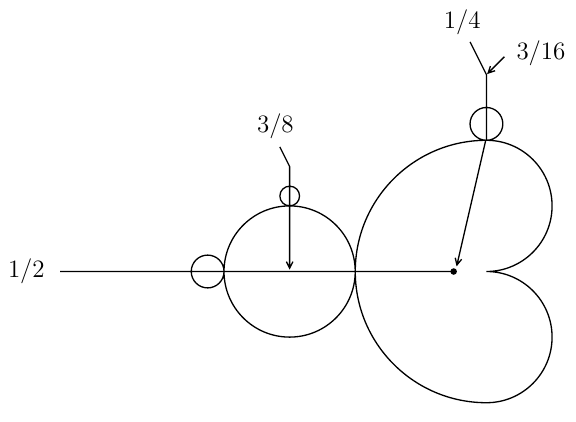}}
\caption{Illustration of directly subordinate parameters: we have $c(3/16)\lhd c(1/4)$, $c(1/4)\lhd c(1/2)$, and $c(3/8)\lhd c(1/2)$. Arrows indicate where the vein to some dyadic parameter terminates at the vein of the dyadic parameter to which it is directly subordinate.}
\label{Fig:DirectlySubordinate}
\end{figure}

\begin{theorem}[Entropy Between Dyadic Hubbard Trees] 
\label{Thm:RelationDyadicTrees} \lineclear 
If $c\lhd c'$ are two dyadic parameters, then 
\begin{itemize}
\item[a)]
$N(n)\le N'(n)$ for all $n$, 
\item[b)]  
$N(n)<N'(n)$ for all sufficiently large $n$, and
\item[c)]
 $\tilde h(c)<\tilde h(c')$. 
\end{itemize}
\end{theorem}

Here $N(n)$ and $N'(n)$ denote the numbers of relevant precritical points of generation $n$ for $p_c$ and for $p_{c'}$. This immediately implies a weak version of the Tiozzo conjecture:

\begin{corollary}[Dyadic Version of Tiozzo Conjecture] 
\label{Cor:DyadicTiozzoConjecture}\lineclear
Every dyadic angle $\theta$ has a neighborhood on which $h$, restricted to dyadic angles, assumes its unique maximum at $\theta$.
\end{corollary}

From here, we could use continuity of $h$ to prove the unrestricted version of the Tiozzo conjecture (which indeed was our strategy in an early version of the proof). We will argue the other way around: it is much easier to deduce the conjecture directly and use this in the proof of continuity.

Let us state another corollary to the Correspondence Theorem~\ref{thm:Corresp}: it is a more precise version than monotonicity of entropy (as stated in Lemma~\ref{Lem:Monotonicity}) because even the lamination is monotone.

%\newpage

\begin{corollary}[Monotonicity of Rational Lamination on Hubbard Tree]
\label{Cor:CountPartsHubTree}
Let $c,c'\in\M$ be two postcritically finite parameters such that $c\prec c'$. Suppose that in the dynamical plane of $p_c$ two periodic or preperiodic rays $R(\phi^-)$ and $R(\phi^+)$ land together at some point in the Hubbard tree of $p_c$, but not on the backwards orbit of the critical value. Then in the dynamical plane of $c'$ the dynamic rays $R(\phi^-)$ and $R(\phi^+)$ also land together and the landing point is in the Hubbard tree of $c'$.
\end{corollary}

\begin{proof}
In the dynamic plane of $p_c$, let $x$ be the landing point of the dynamic ray pair $RP_c(\phi^-,\phi^+)$;  it is by hypothesis in the Hubbard tree of $c$ and it must be a repelling periodic or preperiodic point because $p_c$ is postcritically finite.

Let $W\subset\C$ be the set of parameters $c''$ for which the dynamic rays $R_{c''}(\phi^-)$ and $R_{c''}(\phi^+)$ land together at a repelling periodic or preperiodic point that is not precritical. By Theorem~\ref{thm:Corresp}, the set $W$ is open, and its boundary in $\C$ consists of parameters $c''\in\C\sm\M$ where the critical value is on the forward orbit of one of the rays $R_{c''}(\phi^-)$ and $R_{c''}(\phi^+)$, as well as of parameters $c''\in\M$ where at least one of the two rays lands at a parabolic orbit or at a precritical point. By hypothesis, the set $W$ contains $c$ and is thus non-empty, and any parameter $c''\in\C\sm\ovl W$ must be separated from $c$ by a parameter ray pair with angles on the forward orbit of $\phi^-$ or $\phi^+$. 

However, in the dynamical plane of $c$, all dynamic rays at such angles land at the Hubbard tree, and their configuration shows that there is no parameter ray pair available that, for $c'\succ c$, could bound $c'$ away from $c$. Thus $c'\in W$. 
\end{proof}

%\newpage

\begin{definition}[Dynamical Counterpart to Parameter]
\label{Def:DynamCounterpart} \lineclear
Let $c$ be a postcritically finite parameter and suppose $c'\succ c$. Then a (pre)periodic point $x$ in the dynamical plane of $p_{c'}$ is called the \emph{dynamical counterpart to $c$} in the following case:
\begin{itemize}
\item
if $c$ is preperiodic, then $x$ is the landing point of the preperiodic dynamic rays at the same angles as $c$;
\item
if $c$ is periodic, then $x$ is the landing point of the periodic dynamic rays bounding (in the parameter plane) the subwake of $c$ containing $c'$.
\end{itemize} 
\end{definition}

In the periodic case, $c$ is the center of a hyperbolic component, say $H_c$, and the subwake of $c$ containing $c'$ is bounded by a periodic parameter ray pair landing at $\partial H_c$. The angles of this ray pair are the angles of two rays landing at $x$. For example, the $\alpha$ fixed point is the dynamical counterpart of $c=0$. In the preperiodic case, it is known that all dynamic rays with the same angles of the rays landing at $c$ also land together in the dynamical plane of $c'$.

An equivalent definition is that $x$ is the unique repelling periodic or preperiodic point in the dynamical plane of $c'$ such that the itinerary of $x$ (with respect to the critical point) equals the (upper) kneading sequence of $c$.

\begin{lemma}[Dynamical Counterpart is Characteristic]
\label{Lem:DynCounterpartCharacteristic} \lineclear
Every periodic or preperiodic point that is a dynamical counterpart is characteristic. 
\end{lemma}
\begin{proof}
If, in the dynamical plane of $p_{c'}$, the point $x$ is the dynamical counterpart of a parameter $c$, then $c'\succ c$, and $x$ is the landing point of at least two dynamic rays. If $p_c$ is critically strictly preperiodic, then the parameter rays landing at $c$ bound the wake that $c'$ is in, and this implies that $x$ is a characteristic preperiodic point. If $p_c$ is critically periodic, then the argument is similar, except that the parameter rays do not land at $c$ but at the root of the hyperbolic component containing $c$.
\end{proof}

%\newpage

\begin{lemma}[Directly Subordinate Dyadics]
\label{Lem:DirectlySubordinateDynamics} \lineclear
\looseness-1
If $c$ and $c'$ are two dyadic parameters, then $c\lhd c'$ if and only if there is a postcritically finite parameter $c_*\in\M$ so that $c'$ is the dyadic of least generation within any sublimb of $c_*$, and  $c$ is the dyadic of least generation within a different sublimb of $c_*$ than $c'$.

In this case, denoting the external angles of $c$ and $c'$ by $\theta$ and $\theta'$, respectively, then in the dynamics of $c$ (or any other parameter in the same sublimb of $c_*$) there is a repelling (pre)periodic point $x_*$ that is the landing point of at least three dynamic rays  that separate the dynamic rays at angles $0$, $\theta$, and $\theta'$. The point $x_*$ is the dynamical counterpart to $c_*$. 
\end{lemma}

\begin{proof}
Any two dyadic parameters are endpoints of $\M$, so by the Branch Theorem of the Mandelbrot set there is a unique postcritically finite parameter $c_*$ that contains $c$ and $c'$ in two different of its sublimbs. Let $c_0$ be the unique dyadic of least generation in any of the sublimbs of $c_*$; then $c_*$ is on the long vein of all three of $c_0$, $c$, and $c'$, and it is on the vein of $c_0$. 

The assumption that $c\lhd c'$ means that the vein of $c$ terminates at an interior point of the vein of $c'$, and hence it must terminate at the parameter $c_*$, so $c_*$ is an interior point of the vein of $c'$. Since $c_*$ is also an interior point of the vein of $c_0$, it follows that $c'=c_0$ (two veins can never have more than one point in common). 

Conversely, if $c'$ is the dyadic of least generation in the sublimb of $c_*$    and  $c$ is the dyadic of least generation within a different sublimb of $c_*$ than $c'$, then $c_*$ is in the interior of the vein of $c'$ and the vein of $c$ terminates at $c_*$. This proves the first claim.

For the second claim, we first consider the case that $c_*$ is a Mi\-siu\-re\-wicz-Thurston parameter; it is then the landing point of $s\ge 3$ rational parameter rays, say at angles $\theta_1,\dots,\theta_s$, so that the parameter rays at angles $0$, $\theta$, and $\theta'$ are in different sectors with respect to these parameter rays. Every parameter in any sublimb of $c_*$ has the property that the dynamic rays at angles $\theta_1,\dots,\theta_s$ land together at a repelling preperiodic point, and the claim follows.

If $c_*$ is the center of a hyperbolic component, then the parameter $c$ is in a sublimb at internal angle $p/q\neq 1/2$, and in the dynamical plane of $c$ (or any parameter within the same sublimb) there is a repelling periodic point that is the landing point of $q\ge 3$ dynamic rays that separates the angles $0$, $\theta$ and $\theta'$ so that $\theta$ is in the largest sector not containing the angle $0$. 
\end{proof}

%\newpage

\begin{lemma}[No Extra Branch Point]
\label{Lem:NoExtraBranchPoint} \lineclear
Suppose $c'\in\M$ is a dyadic parameter and $c\in\M$ is the postcritically finite parameter where the vein of $c'$ ends. Let $x$ be the dynamical counterpart of $c$ in $H'$. Then 
the arc $(x,c']\subset H'$ does not contain a branch point.
\end{lemma}
\begin{proof}
Suppose the claim is false and there is a branch point on $(x,c']$. Then the sub-wake of $x$ containing $c'$ also contains a point, say $c''$, in the forward orbit of $c'$ (all endpoints of $H'$ are on the orbit of $c'$). But then $c''$ is the landing point of a dyadic ray of lower generation than the dyadic ray landing at $c'$. This is impossible because in the parameter plane the dyadic ray landing at $c'$ has the lowest possible generation among all dyadic rays in the sub-wake of $c$ containing $c'$ (by hypothesis that the vein of $c'$ terminates at $c$).  
\end{proof}

The main step in proving Theorem~\ref{Thm:RelationDyadicTrees} is a topological surgery on Hubbard trees, as follows: 

%\newpage

\begin{proposition}[Relation Between Subordinate Dyadic Hubbard Trees]
\label{Prop:RelationDyadicTrees} %\lineclear
Let $c\lhd c'$ be two dyadic parameters with external angles $\theta$ and $\theta'$, let $H_c$ be the Hubbard tree of $c$, and let $H \supset H_c$ be the connected hull of the critical orbit and of the orbit of $z(\theta')$. Let $x$ be the branch point of the arcs from $0$ to $c$ and to $z(\theta')$. If $p_c$ is the natural map on $H$, define a map $f\colon H\to H$ as follows: choose a homeomorphism $\rho\colon[x,c]\to[x,z(\theta')]$ fixing $x$ and let 
\[
f(z):=\left\{ \begin{array}{ll} \rho\circ p_c(z) & \text{if $p_c(z)\in[x,c]$} \\ 
p_c(z) & \text{otherwise.} \end{array} \right.
\]
Let $H'$ be the connected hull within $H$ of the orbit of $0$ under $f$. Then $(H',f)$ is the Hubbard tree of $p_{c'}$ (up to isotopy rel branch points and endpoints).
\end{proposition}

\begin{proof}
Since there are no branch points on $[x,c]$ (Lemma~\ref{Lem:NoExtraBranchPoint}), the new map $f$ is a branched covering where $0$ is the unique critical point.

We have a connected tree $H$ containing the critical point $0$, and with respect to $f$ the orbit of $0$ is still finite (it still terminates at the $\beta$ fixed point). Therefore, $(H',f)$ is a finite tree with a continuous self-map, and the dynamics is locally injective except at the critical point $0$. Since there are at most $2$ branches at $0$, the map is globally at most $2:1$. Every endpoint is by definition on the critical orbit, and the tree comes with an embedding into $\C$ that is compatible with the dynamics. Therefore $(H',f)$ is the Hubbard tree of a postcritically finite polynomial in which the critical orbit lands at the $\beta$ fixed point at the desired number of iterations.  Let $c''$ be the corresponding parameter and $\theta''$ be the external angle; we have $\theta''=a''/2^{k'}$.
It remains to prove that $\theta''=\theta'$ and thus $c''=c'$. 

Since $c\lhd c'$, there is a postcritically finite branch point in $\M$, say $c_*$, that separates $c$ from $c'$, and the external angles of $c_*$ are the external angles of $x$ in the dynamical plane of $c$ (Lemma~\ref{Lem:DirectlySubordinateDynamics}). In the dynamical plane of $c''$, the point $x$ has the same external angles because it has the same period and preperiod and the dynamics of the subtree connecting the orbit of $x$ is unaffected by the surgery (except the bit around the critical point that maps past $x$). Hence $\theta''$ is the unique dyadic of least generation that is separated from the angle $0$ by the angles of $x$, and the same is true for $\theta'$.
\end{proof}

\begin{remark}
The fact that $c''=c'$ can also be shown using spiders \cite{Spiders} and Thurston's theorem. Let us topologically extend the map $f:H' \to H'$ to a continuous map on $\C$ as follows. First, we set $f$ to be $p_c$ on the dynamic rays $R({2^{t}\theta'})$ of $p_c$, where $t\in\{0,1,\dots, k'\}$. There are $k'+1$ topological discs in the complement of $H'\cup_{t\ge 0} R({2^{t}\theta'})$ and the map $f$ easily extends to each of them as a homeomorphism. The new map $f$ is a topological polynomial for which $\bigcup_t R({2^{t}\theta'})$ forms an invariant spider. Since this spider is equivalent to a standard invariant spider of $c'$, we get $c'=c''$ by Thurston rigidity \cite{DHThurston,Spiders}.
\end{remark}

%\newpage

\begin{lemma}[Injective Dynamics of Last Edge]
\label{Lem:InjectiveDynamicsLastEdge} \lineclear
In any dyadic Julia set, consider any dyadic angle $\theta=a/2^k$ with $k\ge 1$ and let  $x$ be the point where the arc from $z(\theta)$ to $\alpha$ is attached to the minimal tree connecting all dyadic endpoints of generations $a'/2^{k'}$ with $k'<k$. Then $[z(\theta),x]$ maps injectively for $k$ iterations to an interval $[\beta,y]\subset[\beta,\alpha]$. 
\end{lemma}
\begin{proof}
For every integer $m\ge 0$, let $T_m$ be the minimal tree connecting the $\alpha$ fixed point to all dyadic endpoints of generation at most $m$. Let $f$ be the map on the Julia set. 

The edge $[z(\theta),x]\subset \ovl{T_k\sm T_{k-1}}$ certainly maps forward homeomorphically one generation to an arc $[z(2\theta),x']$, where $x'=f(x)\in f(T_{k-1})$. 
We claim that $[z(2\theta),x']\subset \ovl{T_{k-1}\sm T_{k-2}}$ so that the inductive step applies and completes the proof. 

We first show that 
\begin{equation}
\label{eq:T_k-2InT_k-1}f^{-1}(T_{k-2})\subset T_{k-1}.
\end{equation} 
Indeed, $T_{k-2}$ is the minimal tree connecting all dyadic endpoints of generation at most $k-2$. Consider an endpoint $y$ of the forest $f^{-1}(T_{k-2})$. If $f(y)\in T_{k-2}$ was not an endpoint, so it was connected to  at least two edges in $T_{k-2}$, then $y$ would have to be connected to at least two edges in $f^{-1}(T_{k-2})$, a contradiction. Thus every endpoint of the forest $f^{-1}(T_{k-2})$ is a dyadic endpoint of generation at least $k-1$ and hence $f^{-1}(T_{k-2})\subset T_{k-1}$.

Finally, if $[z(2\theta),x']$ intersects $T_{k-2}$, then $[z(\theta),x]$ intersects $T_{k-1}$, and by hypothesis this intersection is the single point $x$. Hence $[z(2\theta),x']\subset (T_{k-1}\sm T_{k-2})\cup\{x'\}$.
\end{proof}

\begin{remark}
In this lemma, the hypothesis that the polynomial be dyadic was stated only for convenience. All we are using is that the Julia set is path connected (if there are bounded Fatou components, the notation needs minor adjustments). 
\end{remark}

%\newpage

\begin{lemma}[Homeomorphic Preimage of Arc]
\label{Lem:HomeomorphicPreimage} \lineclear
Suppose that $c\lhd c'$ are two dyadic parameters and let $\theta,\theta'$ be their external angles. In the Julia set of $c$, let $x$ be the branch point between $0$, $z(\theta)=c$ and $z(\theta')=:z'$. If $\theta=a/2^k$, then there is a point $y'\in(z',x)$ so that $p_c^{\circ k}\colon [z',y']\to p_c^{\circ k}([z',y'])= p_c^{\circ k}([ c,x])$  is a homeomorphism.
\end{lemma}
\begin{proof} \looseness-1
We know from Lemma~\ref{Lem:InjectiveDynamicsLastEdge} that $[c,x]$ maps homeomorphically for $k$ iterations to a subinterval of $[\beta,\alpha]$; define $y:=p_c^{\circ k}(x)\in(\beta,\alpha]$.
Similarly, there is a point $x'\in[z',0]$ so that $[z',x']$ maps forward homeomorphically for $k'$ iterations (if $\theta'=a'/2^{k'}$). We have 
$x'\in[x,0]$ (or equivalently $x\in[z',x']$) because $c\lhd c'$, i.e.\ $c$ is \emph{directly} subordinate to $c'$.

More precisely, if $c_*$ is the point where the vein of $c$ terminates (Lemma~\ref{Lem:DirectlySubordinateDynamics}), then the external angles of the dynamic rays landing at $x$ are exactly the external angles of the parameter rays bounding the subwake of $c_*$ containing $c$.  Analogously, the same is true for the point $x'$ and the vein of $c'$; let $c'_*$ be this branch point. But since $c\lhd c'$, it follows that $c'_*$ separates $c_*$ from the origin, and hence, by the Correspondence Theorem~\ref{thm:Corresp}, $x'$ separates $x$ from the origin.

%\newpage

\begin{figure}[htbp]
\centerline{
\includegraphics[width=40mm]{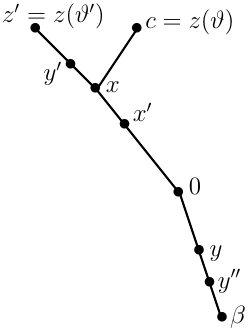}
}
\caption{Illustration of the relative position of various points in the Hubbard tree of $c$ in the proof of Lemma~\ref{Lem:HomeomorphicPreimage}. Note that we do not know or need the relative position between $0$, $y$, and $y''$; in particular, we do not claim $y''\in[\beta,y]$.
}
\label{Fig:HubbardTreeVariousPoints}
\end{figure}

Let $y'':=p_c^{\circ k'}(x)$; then $p_c^{\circ k'}\colon[z',x]\to[\beta,y'']$ is a homeomorphism. Iterating this $k-k'$ further times, the image arc terminates at $\beta$ and at $y$, but it can no longer be injective (the map $p_c^{\circ k}\colon[c,x]\to [\beta,y]$ is a homeomorphism, and $p_c^{\circ k}$ is a local homeomorphism near $x$ because $x$ cannot be on the critical orbit). 

There is a branch $p_c^{-1}\colon[\beta,\alpha]\to[\beta,-\alpha]$; let $[\beta,y''']$ be the image of $[\beta,y]$ under the $k-k'$-th iterate of this branch. 
Observe that $[y''',\beta]\subsetneq[y'',\beta]$  because otherwise $p_c^{\circ (k-k')}$ restricted to $[\beta,y'']$ would have degree $1$. Pulling back $k'$ times we obtain an interval $[z',y']\subset[z',x]$ so that $p_c^{\circ k'}\colon [z',y']\to[\beta,y''']$ and $p_c^{\circ k}\colon[z',y']\to[\beta,y]$ are homeomorphisms, as claimed.
\end{proof}

%\newpage

\begin{proposition}[Injection Between Precritical Points]
\label{Prop:InjectionPrecritical} \lineclear
Given two dyadic parameters $c\lhd c'$, there exists a generation-preserving injection $B$ from the set of precritical points in $[c,-\alpha]$ of $p_c$ to the set of precritical points in $[c',-\alpha]$ of $p_{c'}$. Moreover, $B$ can be taken to satisfy the following properties (A)--(C) for every precritical point $\zeta\in [c,-\alpha]$. 

\begin{itemize}
\item[(A)] For $k\ge 0$ we have $p_c^{\circ k}(\zeta)\in [c,-\alpha]$  if and only if  $p_{c'}^{\circ k}(B(\zeta))\in [c',-\alpha]$. Moreover, if $p_c^{\circ k}(\zeta)\in [c,-\alpha]$, then  
\[
B(p_c^{\circ k}(\zeta))=p_{c'}^{\circ k}(B(\zeta)).
\]
\item[(B)] Assume $c_*$ is a postcritically finite parameter such that $c_* \prec c$ and $c_* \prec c'$. Let $x_*$ and $x'_*$ be the  dynamical counterparts of $c_*$ in the dynamical planes of $c$ and $c'$. Then $\zeta\in [x_*,-\alpha]$ if and only if $B(\zeta)\in [x'_*,-\alpha]$.
\item[(C)] There is a sub-interval $J\subset [c', \alpha]$ such that $p^{\circ k }_{c'}(J)=[-\alpha, \alpha ]$ for some $k>0$ and such that the image of $B$ is in $[c',-\alpha]\setminus J$.
\end{itemize}
\end{proposition}

\begin{proof}
Using Proposition~\ref{Prop:RelationDyadicTrees} and its notation, we may identify $p_{c'}:H_{c'}\to H_{c'}$ with $f:H'\to H'$. 
The point $x_*\in H$ is periodic or preperiodic under iteration of $p_c$ and never visits $[x,c]$; thus $x_*\in H'$ is identified with $x'_*\in H_{c'}$.   
We will now construct a bijection $B$ between precritical points in $[c,-\alpha]$ of $p_c$ and those precritical points in $[z',-\alpha]\setminus [y',x]$ of $f$ for which the orbit never visits $[y',x]$, where $y'$ is specified in Lemma~\ref{Lem:HomeomorphicPreimage} (see Figure~\ref{Fig:HubbardTreeVariousPoints}).
In fact, our bijection will preserve the itinerary with respect to $H\sm\{0\}$, except for $k$ iterations along the orbit from $[c,x]$ to $[\beta,y]=p_c^{\circ k}([c,x])$ (resp.\ from $[z',y']$ to $[\beta,y]$). 

Our proof proceeds by induction on the number of times, say $m$, that an orbit of a precritical point $\zeta$ runs through $[c,x]$ (not counting $\zeta$ itself). We start by those precritical points on $[c,x]$ (for the map $p_c$) that never run through $(c,x)$ again, that is with the case $m=0$. By Lemma~\ref{Lem:HomeomorphicPreimage}, an appropriate branch of $f^{\circ(-k)}\circ p_c^{\circ k}$ sends $[c,x]$ homeomorphically to $[z',y']$; call this branch $\eta\colon[c,x]\to [z',y']$. Then for any $a\in[c,x]$, we have $p_c^{\circ k}(a)=f^{\circ k}(\eta(a))$ and the future orbits of these points under $p_c$ respectively under $f$ coincide as long as the orbits avoid $[c,x]$. Note that all precritical points on $[c,x]$ must have generation at least $k$. We thus obtain an injection, say $B_0$, of precritical points with $m=0$.

Every precritical point $\zeta\in[c,x]$ of generation $n$ and with $m=0$ is the common endpoint of two adjacent sub-intervals of $[c,x]$ that map homeomorphically onto $[c,x]$ after $n$ iterations: we have $p_c^{\circ n}(\zeta)=c$ and $p_c^{\circ (n-1)}(\zeta)=0$, so we can pull the entire interval $[c,x]$ back in two ways (with a choice in the first step) until we end at $\zeta$. The pull-back of the entire interval $[c,x]$ is possible because no critical value can interfere (the critical orbit visits only endpoints of the tree), and the resulting interval is in $[c,x]$ because the Hubbard tree is unbranched on $[c,x)$ and the orbit of $x$ never enters $[c,x)$.

There is an analogous result about precritical points $\zeta'\in[z',y']$ of $f$ with $m=0$ and sub-intervals of $[z',y']$ that map to $[z',y']$. By construction of $y'$, the point $\zeta'$ has generation $n\ge k$, and there is a precritical point $\zeta=\eta^{-1}(\zeta')\in[c,x]$. The point $\zeta$ has a neighborhood, say $I_\zeta\subset [c,x]$, that maps $2:1$ onto $[c,x]$ (the union of the two intervals constructed above), and then $\zeta'$ has a neighborhood $I_{\zeta'}\subset[z',y']$ with $p_c^{\circ k}(I_{\zeta'})=p_c^{\circ k}(I_\zeta)\subset[\beta,y]$.

The bijection $B_0$ of precritical points with $m=0$ thus extends to a bijection between intervals $I\subset[c,x]$ that map homeomorphically onto $[c,x]$ after some number of iterations without visiting $(c,x)$ before, and intervals $I'\subset [z',y']$ that map homeomorphically onto $[z',y']$ after the same number of iterations and without ever visiting $(z',y')$; this bijection respects the number of iterations as well as the order along the intervals within $[c,x]$ and $[z',y']$ (the intervals are obviously disjoint). Denote this bijection of intervals by $B^*_0$.

Now suppose the statement is shown for all precritical points on $[c,x]$ that visit $(c,x)$ at most $m$ times, for some $m\ge 0$; in particular, we have an injection, say $B_m$, from precritical points on $(c,x)$ that map into $(c,x)$ exactly $m$ times, to precritical points on $(z',y')$ that map into $(z',y')$ exactly $m$ times {and never visit $(y',x)$}. Consider any precritical point $\zeta\in(c,x)$ that visits $(c,x)$ exactly $m+1$ times, and let $s$ be minimal such that $p_c^{\circ s}(\zeta)\in(c,x)$. Then there is an interval $I\ni \zeta$ so that $p_c^{\circ s}\colon I\to[c,x]$ is a homeomorphism (same reasoning as above).
Let $I':=B^*_0(I)$. Then $p_c^{\circ s}(\zeta)\in(c,x)$ is a precritical point that visits $(c,x)$ only $m$ times, and $B_{m+1}(\zeta)=\zeta':= f^{\circ(-s)}\circ B_m \circ p_c^{\circ s}(\zeta)$, choosing the branch $f^{\circ(-s)}\colon [y',z']\to I'$.

Since the map $B^*_0$ is injective, different intervals $I$ land in disjoint intervals $I'$, and since $B_m$ is injective by induction, the restriction of $B_{m+1}$ that run through any particular $I$ is injective too, so in total $B_{m+1}$ is injective as claimed. 

This takes care of all precritical points on $[c,x]$, and we still have to deal with those on $[x,-\alpha]$. But those with orbits that never run through $[c,x]$ are unaffected by the changed dynamics, and the injection easily extends to those that map into $[c,x]$ under $p_c$.

It remains to show that the map $B$ satisfies Properties (A)--(C). Let us extend $B_0^*$ to all intervals in $[x,-\alpha]$ that are injective preimages of $[c,x]$ and never run through $[c,x]$ before mapping into $[c,x]$. It is easy to see that $B_0^*(I)\subset I$ for every such interval $I\not \subset [x,c]$ because $[z',y']\subset [z',x]$ and $p_c^{\circ (-1)}([c,x])=f^{-1}([z',x])$, and induction can be applied. Since $x_*$ never visits $[x,c]$, we see that  $x_*\not \in I$ for every maximal pre-image $I$ of $[c,x]$. Also, by construction, $\zeta \in I$ if and only if $B(\zeta)\in I$ for every precritical $\zeta$ and interval $I\not \subset [x,c]$ as above. Therefore, $\zeta$ and $B(\zeta)$ are on the same side of $x_*$. It is clear that $\zeta$ and $B(\zeta)$ have the same return times to $[c,-\alpha]$ and $[z',-\alpha]$ because $p_c^{\circ k}([c,x])=f^{\circ k}([z',y'])$. And the dynamical relation $B(p_c^{\circ k}(\zeta))=p_{c'}^{\circ k}(B(\zeta))$ holds by construction. Thus Properties (A) and (B) hold. To prove (C) set $J$ to be any {iterated} pre-image of $[\alpha,-\alpha]$ so that $J\subset [x,y']$.
\end{proof}

\begin{proof}[Proof of Theorem~\ref{Thm:RelationDyadicTrees}]
From Proposition~\ref{Prop:InjectionPrecritical} we have an injection of precritical points of $p_{c}$ of any given generation on $[-\alpha,c]$ to precritical points of $p_{c'}$ of the same generation on $[-\alpha',c']$, and by claim (B) this injection restricts to the arcs $[-\alpha,\alpha]$ and $[-\alpha',\alpha']$. This immediately proves part a), and part b) follows from claim (C) of Proposition~\ref{Prop:InjectionPrecritical}.

Part c) of Theorem~\ref{Thm:RelationDyadicTrees} also follows from claim (C) of Proposition~\ref{Prop:InjectionPrecritical}:  since $p_{c'}$ is dyadic, every edge of its Hubbard tree, say $H'$, maps over every other in a bounded number of iterations (the Markov partition associated to the edges is irreducible). Therefore, every typical orbit visits $J$ with positive frequency, and the entropy of orbits in $H'$ that are not allowed to visit $J$ is strictly smaller than the full entropy in $H'$. Precritical orbits in $H$ inject to precritical orbits in $H'$ avoiding $J$, and so $H$ has smaller entropy than $H'$.
\end{proof}

%\goodbreak
%\newpage

\section{Irrational Angles and the Tiozzo Conjecture}
\label{Sec:IrrationalTiozzoConj}

In this brief section, we will complete the proof of the Tiozzo Conjecture about local maxima of $h$, and we also describe local minima.

\begin{theorem}[The Tiozzo Conjectures]
\label{Thm:TiozzoConj} \lineclear
The entropy function $h\colon\Circle\to[0,\log 2]$ has the following properties.
\begin{enumerate}
\item[a)] 
Every dyadic angle is an isolated local maximum of the entropy function.

\item[b)]
Conversely, every local maximum of $h$ is dyadic. 
\item[c)] 
Within every wake, the entropy function has a unique global maximum, and it occurs at the unique dyadic of lowest denominator in the wake.

\item[d)]
Within every wake, for each $n$ the function $N_\theta(n)$ assumes its  maximum at the dyadic of least generation (of course, this maximum is not unique).
\end{enumerate}
\end{theorem}

\begin{proof}
Fix a dyadic angle $\theta_0$ and let $I=I(\theta_0)\subset\Circle$ be the open interval of angles $\theta$ for which the combinatorial arcs to $c(\theta)$ intersect the interior of the vein of $c(\theta_0)$ (not the long vein). In other words, if $c_*$ is the endpoint of the vein of $c(\theta_0)$, then $I$ consists of the angles within the same subwake of $c_*$ that $c(\theta)$ is in. Clearly $\theta_0\in I$. Every dyadic angle in $I$ is either directly or indirectly subordinate to $\theta_0$ (where the latter means that there is a finite sequence of dyadic angles ending at $\theta_0$ so that each is directly subordinate to the next). 

By Theorem~\ref{Thm:RelationDyadicTrees} part c), we know that $h$ restricted to dyadic angles in $I$ has its unique maximum at $\theta_0$. 

We claim that for every $\theta\in I$ and every $n\in\N$ we have $N_\theta(n)\le N_{\theta_0}(n)$ and hence $h(\theta)\le h(\theta_0)$. Indeed, if $\theta$ is a dyadic angle, then this is Theorem~\ref{Thm:RelationDyadicTrees} part a). And if not, then there is a dyadic angle $\theta'$ so that $c(\theta)\prec c(\theta')$ and we have $N_\theta(n)\le N_{\theta'}(n)\le N_{\theta_0}(n)$ for all $n$ where the second inequality is again the dyadic argument and the first one is Lemma~\ref{Lem:Monotonicity}.  
Therefore $\theta_0$ is a (weak) local maximum of $N_\theta(n)$ for all $n$ and thus of $h$.

Part a) will follow from the stronger claim that $h$ has a unique global maximum on $I(\theta_0)$, and this occurs at $\theta_0$. We know that any dyadic $\theta\in I(\theta_0)\sm\{\theta_0\}$ has $h(\theta)<h(\theta_0)$. 
If $\theta\in I(\theta_0)$ is such that $c(\theta)\prec c(\theta_0)$, then by Lemma~\ref{Lem:Monotonicity} we have $h(\theta)\le h(\theta_0)$; but in fact we have strict monotonicity because there is some dyadic $\theta'$ with $c(\theta)\prec c(\theta')$ and $h(\theta)\le h(\theta')<h(\theta_0)$. 
And if not $c(\theta)\prec c(\theta_0)$, there is another dyadic $\theta'$ so that $\theta\in I(\theta')$ but $\theta_0\not\in I(\theta')$. We then have $h(\theta)\le h(\theta')<h(\theta_0)$. Therefore, $\theta_0$ is the unique global maximum within $I(\theta_0)\ni\theta_0$. This finishes the proof of the stronger version of claim a).

For part c), consider any hyperbolic component $W$ and let $I$ be the open interval of angles within its wake. Let $\theta_W$ be the unique dyadic of lowest generation within $I$. Then $I(\theta_W)\supset I$, and on this interval $h$ has its unique global maximum at $\theta_W$. Now suppose $W$ is a wake that is not the wake of a hyperbolic component: then either it is one of the subwakes of a Misiurewicz-Thurston parameters, or an irrational wake (bounded by two irrational angles with equal angled internal address). But such wakes are exhausted by wakes of hyperbolic components, so the claim holds for them as well. Part d) also follows.

For claim b), suppose $\theta$ is a local maximum of $h$, and let $I\subset\Circle$ be an interval on which $\theta$ is a global maximum. By monotonicity, we may assume that the $\theta$-ray lands (combinatorially) at an endpoint of $\M$. If $\theta$ is not dyadic, then choose a dyadic angle $\theta'\in I$ with $\theta\in I(\theta')\subset I$. 
Then the unique global maximum of $h$ within $I(\theta')$ is at $\theta'$, so $\theta=\theta'$ is dyadic.
\end{proof}

\begin{remark}
For the record, we observe that along the way we proved that core entropy $\tilde h$ is strictly monotone on arcs before dyadic endpoints: if $c'$ is dyadic and $c\prec c'$, then $\tilde h(c)<\tilde h(c')$ (the general result in Lemma~\ref{Lem:Monotonicity} would only give $\tilde h(c)\le \tilde h(c')$). In fact, there are parameters $c\prec c'$ so that entropy is constant along $[c,c']$; this happens when $c$ and $c'$ are within the same little Mandelbrot sets, for instance within the ``main molecule of $\M$'' (which  was shown in \cite{BruinSchleicher} to be the locus of parameters with zero biaccessibility dimension). 

This implies that core entropy is strictly monotone along all veins, even long veins, except within little Mandelbrot sets: 
if $c\prec c'$ are two postcritically finite parameters, then the Hubbard tree of $c$ (more precisely, its marked points)
can be recovered in the Hubbard tree of $c'$, so all precritical orbits of $c$ are found for $c'$, while there is strictly more choice for $c'$. This choice strictly increases entropy except when $c$ and $c'$ are within the same little Mandelbrot set (or the main molecule with entropy $0$): in the latter case, entropy may be dominated by the non-renormalizable dynamics while the extra choices are added only to the renormalizable dynamics (see Section~\ref{Sub:Renormalizable}). 
\end{remark}

The following is a rather obvious restatement of this result for $\tilde h\colon\M\to[0,\log 2]$.

\begin{corollary}[Local Maxima of $\tilde h$ on $\M$]
\lineclear
The function $\tilde h\colon\M\to[0,\log 2]$ has local maxima exactly at parameters $c(\theta)$ with $\theta=k/2^q$, and these are all isolated.
\end{corollary}

The graph in Figure~\ref{Fig:EntropyGraph} also shows local minima, and these can be classified relatively easily; we are grateful to Steffen Maass for his questions.

\begin{theorem}[Local Minima of Entropy]
The function $h\colon\M\to[0,\log 2]$ has an isolated local minimum at $\theta$ if and only if there are two angles $\theta_1,\theta_2$ with $0<\theta_1<\theta<\theta_2<1$ so that all three parameter rays at angles $\theta,\theta_1,\theta_2$ land at a common parameter; in this case all angles $\theta,\theta_1,\theta_2$ are rational with odd denominators that are all equal, and the landing point of these rays is a Misiurewicz-Thurston parameter.

The function $\tilde h\colon\M\to[0,\log 2]$ has no isolated local minimum.
\end{theorem}

\begin{proof}
Since entropy is monotone along the combinatorial arc $[0,c]$ for every $c\in\M$, the entropy function $\tilde h$ cannot have isolated local minima when $\tilde h(c)>0$. Since $\tilde h^{-1}(0)$ is connected (the main molecule of $\M$), there is no isolated local minimum at all.

By strict monotonicity along veins, except within little copies of $\M$, it is obvious that $h$ has an isolated minimum at every angle $\theta$ for which there are angles $\theta_1,\theta_2$ as claimed. Conversely, if $h$ has an isolated local minimum, then all angles associated to the combinatorial arc $[0,c(\theta)]$ must avoid a neighborhood of $\theta$.
For sufficiently small $\eps_1>0$, $\eps_2>0$ so that $\theta-\eps_1\in\Q$ and $\theta+\eps_2\in\Q$, the Branch Theorem applied to $c(\theta-\eps_1)$ and $c(\theta+\eps_2)$ shows that either $c(\theta)$ is a Misiurewicz-Thurston parameter and $\theta=p/q$ with odd $p$ and even $q$, or $c(\theta)$ is a boundary point of a hyperbolic component of $\M$. In the latter case, every neighborhood of $\theta$ contains angles for which the corresponding rays land at the boundary of the same hyperbolic component, and these have the same entropy (even the same number of relevant precritical leaves).
\end{proof}

\reminder{Should we also discuss non-isolated local minima? These occur at little Mandelbrot sets, but this would require the statement that entropy is constant within every little Mandelbrot set that has positive entropy at the root.}

%\end{document}

%\newpage

\section{Continuity of Entropy}
\label{Sec:Continuity}

Many of our estimates will involve ``radial growth of entropy'': that is, comparing $\tilde h(c')-\tilde h(c)$ for parameters $c'\succ c$. 
A fundamental case will be when $c'$ is a dyadic endpoint of $\M$ and $c$ is the parameter where the vein of $c'$ terminates. We start by comparing the corresponding Hubbard trees. 
Since we count pre-critical points between $\alpha$ and $-\alpha$, we add the points $\alpha$ and $-\alpha$ to the set of vertices of the Hubbard trees 
(if they are not already there).

%\newpage 

\begin{lemma}[Marked Points in Related Hubbard Trees]
\label{Lem:MarkedPointsSubset_new} \lineclear
%\looseness-1
Suppose $c'\in\M$ is a dyadic parameter and $c\in\M$ is the postcritically finite parameter where the vein of $c'$ ends. We assume that $c\not= 0$. Denote by $H$ and $H'$ the Hubbard trees of $p_c$ and $p_{c'}$. If $-\alpha\not\in  H$, then extend $H$ to $-\alpha$ by adding an extra edge. Then $H$ and $H'$ are related as follows. Let $V$ be the union of the postcritical points, $\{\alpha, -\alpha \}$, and branch points in $H$. Denote by $x$ the dynamical counterpart of $c$ in $H'$. Let $V'$ be the union of the postcritical points, $\{\alpha,-\alpha\}$, branch points, and points on the orbit of $x$ in $H'$, and let $V'_*$ be the union of $\{\alpha, -\alpha\}$, the branch points, and the points on the orbit of $x$ in $H'$. Then 
\begin{itemize}
\item[(A)] $V'\sm V'_*$ are exactly the endpoints of $H'$; and
\item[(B)] there is a bijection $J\colon V\to V'_*$ so that 
\begin{itemize}
\item[(1)] $J(p_c(v))=p_{c'}(J(v))$ for all $v\in V$;
\item[(2)] if $v\in V$ is the landing point of a ray with angle $\phi$, then $J(v)\in V'$ is also the landing point of the ray with angle $\phi$; \item[(3)]  $J(c)=x$; and 
\item[(4)] if a ray pair $RP(\phi_1,\phi_2)$ lands at a non-precritical point of $H$ so that $RP(\phi_1,\phi_2)$ separates two vertices $v_1,v_2\in H$, then $RP(\phi_1,\phi_2)$ separates $J(v_1)$ and $J(v_2)$ in the dynamical plane of $p_{c}$.
\end{itemize}
\end{itemize}
\end{lemma}

\begin{remark}
Note that the tree $H$ contains $-\alpha$ unless $c$ is immediately satellite renormalizable (see Lemma~\ref{Lem:RenormMinimalPeriod}); in all other cases $-\alpha$ is in the connected hull of the critical orbit. 

Observe also that every $v\in V$ is either the landing point of a periodic or preperiodic dynamic ray, or on the orbit of $c$.
\end{remark}

\begin{proof}
The set of marked points of the Hubbard tree $H$ are the postcritical points, which include the endpoints, the branch points,  and $\{\alpha, -\alpha\}$. They form the set $V$, and we show that these points exist, with the same combinatorics, for all parameters $c''\succ c$, including $c'$. 

Let $v$ be a branch point of $H$. Then $v$ is the landing point of at least three dynamic rays, all periodic or preperiodic, and the rays at the same angles land at a common point for all parameters $c''\succ c$, in particular for $c'$, by Corollary~\ref{Cor:CountPartsHubTree}. It is not hard to see that the tree $H'$ has (at least) as many branches at the corresponding point as $H$ does at $v$. This defines the map $J$ on the set of branch points on $H$ in a natural way.  The case $v\in\{\alpha, -\alpha\}$ is treated similarly.

All further marked points of $H$ are the critical value $c$ and its finite forward orbit. In the dynamics of $H'$, there is the dynamical counterpart $x$ of $c$, see Definition~\ref{Def:DynamCounterpart}. Let us set $J(p_c^{\circ n}(c)):=p_{c'}^{\circ n}(x)$ and show that $J$ is well defined and satisfies the requirements (1)--(4).

If $c$ is a Misiurewicz-Thurston parameter, then $p^{\circ k}_c(c)$ is the landing point of at least two external rays (because $c\prec c'$ is not an endpoint). Moreover, $p^{\circ k}(c)$ is the landing point of a ray $R(\phi)$ if and only if $R(\phi)$ lands at $J(p^{\circ k}(c))=p'^{\circ k}(x)$ in the dynamical plane of $p'$. Thus $J$ is well defined and satisfies (1)--(4) in the Misiurewicz-Thurston case.

The other case is that $c$ is the center of a hyperbolic component $W\subset\M$. In this case, no ray lands at $p^{\circ k}(c)$ for $k\ge 0$ because these points are in the Fatou set. By definition, $J(c)=x$ is the landing point of a periodic ray pair, say $RP(\phi_-,\phi_+)$, so that in parameter space the ray pair at the same angles bounds the subwake of $c$ containing $c'$. Denote by $c_*\in \partial W$ the landing point of the parameter ray pair $RP(\phi_-,\phi_+)$. For a parameter $w\in \ovl W$, let $\gamma_w$ be a periodic point on the unique non-repelling orbit; there is a unique continuous choice so that $\gamma_c=c$.  Then  $\gamma_{c_*}$ is the landing point of the dynamic ray pair $RP(\phi_-,\phi_+)$. Observe also that the cycle $\{p_w^{\circ k}(\gamma_w)\}_{k\ge 0}$ for $w\in W\cup\{c_*\}$ does not cross any ray pair $RP(\phi_1,\phi_2)$ as in the requirement (4) of $J$. This proves claim (4) because the rays landing at non-precritical points of $H$ plus rays $R(\phi_-)$ and $R(\phi_+)$ are stable in the subwake of $c$ containing $c'$. The requirements (1)--(3) are immediate.

We get a natural injection $J\colon V\to V'$. None of the image points are endpoints: the images of branch points are branch points, the images of $\alpha, -\alpha$ are $\alpha, -\alpha$, and the image of $x$ is on $[0,c']$ and not an endpoint, so the forward iterates of $x$ cannot be endpoints either. Hence $J(V)\subset V'_*$. Since there is no branch point on $(x,c']$ (Lemma~\ref{Lem:NoExtraBranchPoint}), all branch points of $H'$ are in the connected hull of the orbit of $x$. Therefore, $J:V'\to V'_*$ is a bijection. This completes the proof.
\end{proof}

%\enlargethispage{5mm}

%\newpage

\begin{lemma}[Corresponding Dynamics on Edges]
\label{lem:HintoH'} \lineclear
\looseness-1
Let $H$ and $H'$ be the Hubbard trees of $c$ and $c'$ as in Lemma~\ref{Lem:MarkedPointsSubset_new}. Let $V$ and $V'$ be the vertex sets of $H$ and $H'$ (again as in Lemma~\ref{Lem:MarkedPointsSubset_new}). Then the bijection $J\colon V\to V'_*$ extends to an injection of edges in $H$ to edges in $H'$. The image of $H$ under this map on edges is the connected hull in $H'$ containing the orbit of $x$. 

Let $e'_0\subset H'$ be the edge that contains the critical point in its interior. 
The critical point of $p_c$ is in the interior of $J^{-1}(e'_0)\subset H$ if $c$ is pre-periodic, and it is on the boundary of $J^{-1}(e'_0)\subset H$ if $c$ is periodic.
Moreover, an edge $e_1\subset H$ covers once (resp.\ twice) an edge $e_2\subset H$ under $p_c$ if and only if the edge $J(e_1)\subset H'$ covers once (resp.\ twice) an edge $J(e_2)\subset H'$ under $p_{c'}$.

If there are two edges $e'_1,e'_2\subset H'$ with $e'_1\subset J(H)$ and $e'_2\not\subset J(H)$ so that $p_{c'}(e'_1)\supset e'_2$, then $e'_1=e'_0$ and $e'_2=[c',x]$.
\end{lemma}
\begin{proof}
If an edge $e\subset H$ connects two vertices $v_1, v_2\in V$, then by Lemma~\ref{Lem:MarkedPointsSubset_new} the vertices $J(v_1), J(v_2)\in V'$ are adjacent; i.e.\ $J(v_1)$ and $J(v_2)$ are connected by an edge that we defined to be $J(e)$. This extends $J$ to an injection of edges in $H$ to edges in $H'$. Clearly, $J(H)$ is the connected hull of $V'_*=J(V)$. 

Since $p_{c'}$ is a Misiurewicz-Thurston parameter, the critical point of $p_{c'}$ is in the interior of an edge; we denote this edge by $e'_0=[a,b]\subset H'$. Then both pre-images of $x$ must be in $[a,b]$ because $(x,c']$ contains no branch point (Lemma~\ref{Lem:NoExtraBranchPoint}) and thus no marked point other than $c'$.

If $c\in H$ is periodic, then so is $x\in H'$, and at least one of the two pre-images of $x$ must be in $V'$, so this point must be in $\{a,b\}$. Thus one of $p_{c'}(a)$ and $p_{c'}(b)$ equals $x$. Since $x=J(c)$, one of $J^{-1}(a)$ or $J^{-1}(b)$ is the critical point in $H$.

If $c\in H$ is pre-periodic, then the critical point $0\in H$ is not a marked point, so it is an interior point of some edge, say $e_0$. All marked points in $H$ are landing points of (pre)periodic dynamic rays, and the map $J$ respects their external angles, and this implies that $J$ must send $e_0$ to $e'_0$ (the critical point must be accessible by two rays with angles that differ by $1/2$).

It is now easy to see that $J$ respects the dynamics of edges in $H$ and in $H'$, except that $e'_0$ covers, in addition, twice $[c',x]$. Indeed, if an edge $[v,w]\in H$ is different from $e_0$, then  $[J(v),J(w)]\not= e'_0$; thus $p_{c}$ maps $[v,w]$ homeomorphically onto $[p_{c}(v),p_{c}(w)]$, while $p_{c'}$ maps $[J(v),J(w)]$ homeomorphically onto $[p_{c'}(J(v)),p_{c'}(J(w))]=[J(p_c(v)),J(p_c(w))]$.

It remains to analyze the edges $e'_0=[a,b]$ and $e_0:=[J^{-1}(a),J^{-1}(b)]$; this is done similarly: 
the arcs $[a,0]$ and $[0,b]$ cover homeomorphically $[p_{c'}(a),c']$ and $[c',p_{c'}(b)]$ respectively, while $[J^{-1}(a),0]$ and $[0,J^{-1}(b)]$ respectively cover  $[p_{c}(J^{-1}(a)),c]$ and $[c,p_{c}(J^{-1}(b))]$. 

Finally, if $e'_1\subset J(H)$ but $p_{c'}(e'_1)\not\subset J(H)$, then $p_{c'}\colon e'_1\to p_{c'}(e'_1)$ cannot be a homeomorphism (because $J(H)$ is connected), so $e'_1$ must be the unique edge containing the critical point, so $p_{c'}(e'_1)\ni c'$; since $(x,c')$ contains no marked point of $H'$, we have $p_{c'}(e'_1)\sm J(H)=(x,c']$. 
\end{proof}

%\newpage

Next we need a combinatorial estimate. A ``combinatorial pattern of length $n$ with gap size $s$'' is a finite sequence of integers $(j_1,j_2,\dots,j_m,n)$ with $1\le j_1< j_2 < \dots < j_m< n$ and $j_{i+1}-j_i\ge s$ and $n-j_m\ge s$.

\begin{lemma}[Number of Combinatorial Patterns]
\label{Lem:CombinatorialPatterns} \lineclear
The number of combinatorial patterns of length $n$ with gap size $s$ is at most $e^{(n/s)\log(s+1)}$.
\end{lemma}
\begin{proof}
The number of combinatorial patterns equals the number of binary sequences of length $n$ where two consecutive digits $1$ have distance at least $s$, and so that the final digit is a $1$. Write $n=ks+r$ with $r<s$. Then each block of $s$ consecutive entries has $s+1$ possibilities because it has at most a single $1$, and the last block has $r-1$ digits $0$ followed by a $1$, so it has $1$ possibility. The number of combinatorial patterns is thus at most $(s+1)^k= e^{k\log(s+1)}\le e^{(n/s)\log(s+1)}$.
\end{proof}

%\newpage

In the proof of Proposition~\ref{Prop:BoundEntropyIncrease}, it will be convenient to define ``relevant precritical points'' as precritical points on $[\alpha,\beta]$, i.e.\ precritical leaves separating the two fixed points $\alpha$ and $\beta$ or their corresponding leaves in the lamination (rather than separating $\alpha$ from $-\alpha$ as before). We thus start by showing that this will not affect the value of the entropy.

%\newpage

\begin{lemma}[Different Counts Yield Identical Entropy]
\label{lem:DifferCounts}
\lineclear
In any invariant quadratic lamination, let $N_1(n)$ be the number of precritical leaves that separate $\alpha$ from $-\alpha$, and let $N_2(n)$ be the number of precritical leaves of generation $n$ that separate $\alpha$ from $\beta$. Then
\[
\limsup_n \frac 1 n \log N_1(n) = \limsup_n \frac 1 n \log N_2(n)
\;;
\]
in other words, both counting functions define the same entropy. 
\end{lemma}
\begin{proof}
Since $-\alpha\in[\alpha,\beta]$, we clearly have $N_1(n)\le N_2(n)$. To show the converse, we claim that $N_2(n) \le N_1(n)+N_1(n-1)+ N_1(n-2)+\dots$.

To see this, denote $\alpha_0:=\alpha$ and, recursively, $\alpha_{k+1}$ to be the unique preimage of $\alpha_k$ on $[\alpha,\beta]$, so $\alpha_1=-\alpha$. Then $[\alpha_{k+1},\alpha_k]$ maps homeomorphically onto $[\alpha_k,\alpha_{k-1}]$ and $[\alpha,\beta]=\bigcup_{k\ge 0}[\alpha_{k},\alpha_{k+1}]$. If $N_{[\alpha_k,\alpha_{k+1}]}(n)$ denotes the number of precritical points of generation $n$ on $[\alpha_k,\alpha_{k+1}]$, then we have $N_{[\alpha_{k+1},\alpha_{k+2}]}(n)=N_{[\alpha_k,\alpha_{k+1}]}(n-1)$ and $N_{[\alpha_{0},\alpha_{1}]}(n)=N_1(n)$ and indeed $N_2(n) = N_1(n)+N_1(n-1)+N_1(n-2)+\dots$.  If $N_1(n)\le C e^{(h+\eps)n}$ for all $n$, then $N_2(n)\le C n e^{(h+\eps)n}$. Therefore, $N_1$ and $N_2$ define the same entropy.
\end{proof}

%\newpage

\begin{proposition}[Bound on Entropy Increase]
\label{Prop:BoundEntropyIncrease} \lineclear
Suppose $c_1\prec c_2$ and $[c_1,c_2]$ is a (combinatorial) arc in $\M$ such that $0\le \tilde h(c_2)-\tilde h(c_1)\le \eps$. Then there is an $s>0$ with the following property: if $[c,c']$ is a dyadic vein of generation at least $s$ that terminates at $c\in [c_1,c_2]$, then $\tilde h(c')-\tilde h(c_2)\le \eps$. \end{proposition}
\reminder{Would it be more consistent to use $\ovl m$ for $s$? }

\begin{proof}
Consider first the case $c\not =0$. Set $h:=\tilde h(c_2)$. There is a $C>0$ so that all $c\in [c_1,c_2]$ satisfy $N_c(n)\le N_{c_2}(n)\le Ce^{(h+\eps/2)n}$ for all $n$ (monotonicity, Lemma~\ref{Lem:Monotonicity}). We may suppose that $s$ is large enough so that $2C\le e^{(h+\eps/2)s}$. Let $s'\ge s$ be the generation of $c'$. 

As before (see Lemma~\ref{Lem:MarkedPointsSubset_new}), let $H'$ be the Hubbard tree of $p_{c'}$, which is dyadic, and let $H$ be the Hubbard tree of $p_c$, which is postcritically finite (as endpoint of a dyadic vein in $\M$), so both Hubbard trees exist and are finite. Since in this proof we count pre-critical points in $[\alpha,\beta]$ we add to $H$ the arc, say $[\beta',\beta]$,  connecting $\beta$ to $H$; denote by 
$H_\beta:= H\cup [\beta',\beta]$ the extended Hubbard tree. Clearly, $H_\beta$ is $p_c$-invariant.

Let $x\in H'$ be the dynamical counterpart of $c$. Recall that $x$ is a characteristic periodic or preperiodic point in the sense that the entire orbit of $x$ is contained in the closure of the component of $H'\sm\{x\}$ that contains $0$ (Lemma~\ref{Lem:DirectlySubordinateDynamics}). 
In particular, the orbit of $x$ is disjoint from $(x,c']$. It follows that any connected component $I$ of $p_{c'}^{\circ (-n)}([x,c'])$ within $H'$ is either contained in $[x,c']$ or intersects it at most in $\{x\}$. (Otherwise, $x$ would be in the interior of $I$ and after $n$ iterates $x$ would be mapped into $(x,c']$, but $H'$ does not have a branch point on $(x,c']$ by Lemma~\ref{Lem:NoExtraBranchPoint}).

We know from Lemma~\ref{Lem:InjectiveDynamicsLastEdge} that
\begin{equation}
p_{c'}^{\circ s'}([x,c']) \subset[\alpha,\beta]
\label{Eq:IterateOf[x,ctilde]}
\end{equation}
and, moreover, the orbit of $p_{c'}^{\circ i}([x,c'])$ for $i\in \{0,1,\dots , s'-1\}$ does not contain $0$.

Before we finish the proof of Proposition~\ref{Prop:BoundEntropyIncrease}, we need to define a few terms, and we need a lemma. 

\looseness+1
By a \emph{maximal preimage of $[x,c'] \subset H'$ of generation $n>0$} we 
mean a connected component $I$ of $p_{c'}^{\circ (-n)}([x,c'])$ so that $I\not\subset p_{c'}^{\circ(-i)}[x,c']$
for every $i\in\{1,2,\dots,n\}$; equivalently, $p_{c'}^{\circ i}(I)\not \subset [x,c']$ for all $i<n$ {(we just proved that this implies that  $p_{c'}^{\circ i}(I)\cap[x,c']\subset\{x\}$} ). 
We denote by $M$ the set of all maximal preimages of $[x,c']$ that are in $[\alpha,\beta]$.

An \emph{itinerary} of $I$ will be a sequence $s_{0}s_1\dots s_{n-2}\in\{\0,\1\}^{n-1}$ where each $s_{i}$ describes the connected component of $H'\setminus \{0\}$ containing $p_{c'}^{\circ i}(I)$ (labeled for instance so that the critical value is in the component with label $\1$). {Our construction assures that every itinerary is well-defined (the immediate preimage of $I$ contains $0$, but further preimages do not because $[x,c')$ does not contain postcritical points, and $c'$ is not periodic). }
%\newpage

%\enlargethispage{30mm}

\pagebreak%[3]

\begin{lemma}
\label{lem:PreCritInterv}
%\reminder{For Dierk to read.}
There is an itinerary preserving injection from\nopagebreak
\begin{itemize}\nopagebreak
\item the set $M$ of maximal preimages of $[x,c'] \subset H'$ of generation $n$ to
\item the set of precritical points of $f:H_\beta\to H_\beta$ of generation $n$.
\end{itemize}
Moreover, a maximal preimage of $[x,c'] \subset H'$ belongs to the interval $[\alpha,\beta]\subset H'$ if and only if the corresponding precritical point belongs to $[\alpha,\beta]\subset H_\beta$.
\end{lemma}

\begin{proof}
The map $J:H\to H'$ from Lemma~\ref{lem:HintoH'} extends naturally into $J:H_\beta \to H'$ such that the new map
embeds the sets of vertices and edges of $H_\beta$ into the sets of vertices and edges of $H'$.

Every interval $I\in M$ is uniquely characterized by a sequence $\ovl \tau'=(\tau'_0,\eps_{0},\tau'_1,\eps_{1},\dots, \eps_{n-2},\tau'_{n-1})$ so that $\tau'_i$ is the edge of $H'$ containing $p_{c'}^{\circ i}(I)$ and $\eps_i\in \{\0,\1\}$ describes the connected component of $H'\sm \{0\}$ containing $p_{c'}^{\circ i}(I)$, labeled for instance so that the critical value is in the component with label $\1$. Of course, $\eps_i$ is determined by $\tau_i$ unless  $\tau_i=e'_0$; similarly, $\eps_i$ and $\tau_{i+1}$ uniquely determine $\tau_i$. 

The sequence $\ovl \tau'$ is subject to the following conditions:
\begin{itemize}
\item
$p_{c'}(\tau'_i)\supset \tau'_{i+1}$ \;;
\item
all $\tau_i\neq [c',x]$ (by the condition of ``maximal preimage'') \;;
\item
$\tau_0\subset [\alpha,-\alpha]$\;; and \item $\tau'_{n-1}=e'_0$.
\end{itemize}
 
 Since all $\tau_i\neq [c',x]$ and  $\tau_0\subset [\alpha,-\alpha]$, all $\tau_i$ are in the image of $J$. Define the sequence 
\[
\ovl \tau=J^{-1}(\ovl \tau'):= (J^{-1}(\tau'_0), \eps_0, J^{-1}(\tau'_1), \eps_1,\dots , J^{-1}(\tau'_{n-1})).
\] 
By Lemma~\ref{lem:HintoH'}, $ J^{-1}(\tau'_{n-1})=e_0$ while $J^{-1}(\tau'_0)\subset [\alpha,-\alpha]$.
 
 Since $e_0$ contains the critical point (in its interior or in its boundary) the sequence $\ovl \tau$ determines a unique relevant pre-critical point $\ell$ such that $p_c^{i}(\ell)$ is contained in the intersection of $\tau_i$ with component of $H_\beta\setminus \{0\}$ labeled by $\eps_i$. It is straightforward that a different choice of $I$ leads to a different choice of $\ovl \tau'$, which leads to a different choice of $\ell$. And the itinerary of $I$, which is $\eps_0\eps_1\dots \eps_{n-2}$, is preserved.
\end{proof}

%\newpage

{Now we continue the proof of Proposition~\ref{Prop:BoundEntropyIncrease}}.
By Lemma~\ref{lem:PreCritInterv} the number of intervals in $M$ of generation $n$ that are in $[\alpha,\beta]$ is bounded above by $Ce^{(h+\eps/2)n}$.

In order to bound the number of relevant precritical points of any generation $n$ in the dynamics of $p_{c'}\colon H'\to H'$ (these are, by definition, the precritical points in $[\alpha,\beta]$), consider any relevant precritical point $\ell$ of generation $n$. 
Let $j_1<j_2<\dots <j_m=n$ be the set of all iterates so that $p_{c'}^{\circ j_i}(\ell)\in [x,c']$.
Then  $\ell$, $p_{c'}^{\circ (j_1+s')}(\ell)$, $p_{c'}^{\circ(j_2+s')}(\ell)$,\dots , $p_{c'}^{\circ(j_{m-1}+s')}(\ell)$ are within $[\alpha,\beta]$; compare \eqref{Eq:IterateOf[x,ctilde]}. This also implies that $j_{i+1}-j_i\ge s'$. 

Let $I_0,I_1,\dots , I_{m-1}\in M$ be the unique intervals in $M$ containing, respectively, 
$\ell, p_{c'}^{\circ(j_1+s')}(\ell),p_{c'}^{\circ(j_2+s')}(\ell),\dots , p_{c'}^{\circ(j_{m-1}+s')}(\ell)$; their respective generations are $j_1, j_2-j_1-s',\dots , j_{m}-j_{m-1}-s'$. 
We claim that $\ell$ has itinerary  $\s=s_0s_1s_2\dots s_{n-2}$ of $\ell$ as follows:
\begin{itemize}
\item $s_0s_1\dots s_{j_1-2}$ is the itinerary of $I_0$;
\item $s_{j_i}s_{j_i+1}\dots s_{j_i+s'-1}$ is the kneading sequence of $c'$; 
\item $s_{j_{i}+s'}s_{j_{i}+s'+1}\dots s_{j_{i+1}-2}$ is the itinerary of $I_i$; and
\item all $s_{j_{i}-1}$ are arbitrary in $\{\0,\1\}$.
\end{itemize} 
We justify this as follows: 
$p_{c'}^{\circ(j_1-1)} $ maps $I_0$ homeomorphically, while  $p_{c'}^{\circ j_1}\colon I_0\to [x,\tilde c]$ is a $2:1$-map, so the first $j_1-2$ iterates do not contain $0$ and all points in $I_0$ have the same entries in their itineraries up to entry number $j_1-2$. 

Since $p_{c'}^{\circ j_i}(\ell)\in [x,c']$, the next iterates are the same as for $[x,c']$ and, in particular, for $c'$, hence equal to the kneading sequence of $c'$, at least before $c'$ lands at the $\beta$ fixed point, that is for $s'-1$ iterations. 

The iterate $p_{c'}^{\circ (j_i+s')}(\ell)$ is by definition in $I_i$, and this interval travels forward homeomorphically until it covers $0$, which is the iteration before it reaches $[x,c']$ the next time; since the latter is at iterate $j_{i+1}$, the itinerary of $\ell$ coincides with that of $I_i$ until position $j_{i+1}-2$ (analogous to the beginning). 
In the subsequent iterate, the image interval $p_{c'}^{\circ (j_{i+1}-j_i-1)}(I_i)$ contains $0$, so both entries in the itinerary are possible.

Now consider the set of all precritical points in $[\alpha,\beta]$ of generation $n$ corresponding to a particular combinatorial pattern $(j_1,j_2,\dots,j_m,n)$.
We just showed that in order to determine the itinerary of $\ell$ we only need to specify $s_{j_1-1},s_{j_2-1},\dots , s_{j_m-1}\in \{\0,\1\}$ as well as the intervals $I_0,I_1,\dots , I_{m-1}$ as above; their numbers we estimated {in Lemma~\ref{lem:PreCritInterv}}. Therefore, the total number of precritical points with pattern $(j_1,j_2,\dots,j_m,n)$ is at most 
\[
2 C e^{(h+\eps/2)(j_1-1)} \left(\prod_{i=1}^{m-1} 2C e^{(h+\eps/2)(j_{i+1}-j_i-s'-1)}\right)\le2C e^{(h+\eps/2)n}
\]
because $2Ce^{-(h+\eps/2)s'}\le 1$ by hypothesis.

Since the number of combinatorial patterns is at most $e^{(n/s')\log(s'+1)}$ (Lemma~\ref{Lem:CombinatorialPatterns}), it follows that $N_{\tilde c}(n) \le 2 C e^{n(h+\eps/2+\log(s'+1)/s')}$.

Therefore
\begin{align*}
\tilde h(\tilde c) 
&\le \limsup_n \frac{1}{n}\left(\log 2 +\log C + n (h+\eps/2)+\frac{n}{s'}\log(s'+1) \rule{0pt}{11pt} \right)  
\\
&\le h+\eps/2+ \frac{\log (s'+1)}{s'} \le h+\eps
\;
\end{align*}
if $s'\ge s$ is sufficiently large.

Consider now the case $c=0$. We claim that $h(c')<\varepsilon$ if $s$ is sufficiently big. Denote by $K'$ the filled in Julia set of $p_{c'}$. Suppose that $K'\setminus \{\alpha\}$ consists of $q$ connected components (this is equivalent to $c'$ being in the primary $p/q$ limb of the Mandelbrot set for some $p$ coprime with $q$), we enumerate them as $K_1,K_2,\dots ,K_{q}$ such that $K_1$ contains the critical value, $K_q$ contains the critical point, and $p_{c'}$ maps $K_i$ homeomorphically onto $K_{i+1}$ for all $i<q$. Observe that $\beta\in K_q$ and there is a unique $c_i\in p_{c'}^{q-i}(\beta)$ such that $c_i\in K_i$. Since $c'$ is the dyadic endpoint of smallest generation in a limb of $c=0$ we see that $c_1$ is the critical value of $p_{c'}$, in particular $q\ge s$. Thus $H'=\cup _{i\le q} [\alpha, c_i]$ is a star-like tree. The associated transition matrix of $p_{c'}\colon H'\to H'$ is 
\[M=\left(\begin{matrix}
    0      & 1 & 0 &0 & \dots &0 &0 \\
    0       & 0 & 1&0 & \dots &0 & 0 \\
    0       & 0 & 0&1 & \dots &0 & 0 \\
    \hdotsfor{7} \\
   0      & 0& 0 &0 & \dots & 0 &1\\
   2      & 0& 0 &0 & \dots & 0 &1
\end{matrix} \right)\]
It is easy to see that if $q\ge s$ is big enough, then the leading eigenvalue of $M$ is close to $1$, thus its logarithm is close to $0$.  
\end{proof}

%\newpage

\begin{theorem}[Continuity of Entropy at Hyperbolic Components]
\label{Thm:ContEntropyHypComps} \lineclear
The core entropy function $h$ is continuous at all angles that are associated to hyperbolic components of $\M$.
\end{theorem}
\begin{proof}
Let $W$ be a hyperbolic component of $\M$ of some period $n$ and suppose an angle $\theta$ is associated to $W$, in the sense that the parameter ray at angle $\theta$ lands at $\partial W$ (in fact, all we are using is that the ray accumulates at $\partial W$; we are not assuming the known fact that all such rays actually land). 

Let $c_1$ be the root of $W$ and $c_2$ the bifurcation point in $\partial W$ of the period $2n$ component. The combinatorial arc $[c_1,c_2]$ consists of the two internal rays of $W$ connecting the center, say $c_0$, to $c_1$ and to $c_2$. 

Every sublimb $L$ of $W$ has a leading dyadic, say $c_L$, at which the entropy within $L$ is maximal (Theorem~\ref{Thm:TiozzoConj}). Unless $L$ is the $1/2$-limb, the vein of $c_L$ terminates at $c_0$, so for given $\eps>0$, by Proposition~\ref{Prop:BoundEntropyIncrease} there are only finitely many limbs of $W$ in which the entropy exceeds $\tilde h(c_1)+\eps$. This immediately implies continuity of $h$ at all irrational angles $\theta$ associated to $W$. 

If $\theta$ is a rational angle associated to $W$, then either $c(\theta)$ is the root of $W$, or $c(\theta)$ is the parameter where $W$ bifurcates to a component $W'$ with period a proper multiple of $n$. In the latter case,  Proposition~\ref{Prop:BoundEntropyIncrease} implies continuity of $h$ at $\theta$ among all rays that are not in the wake of $W'$. The same argument, applied to $W'$, implies continuity at $\theta$ among all rays in the wake of $W'$. 

We finally have to discuss the case that $\theta$ lands at the root of $W$. Continuity among rays in the wake of $W$ is handled once again as before. 

The only case left is when $W$ is a primitive hyperbolic component of period $n\ge 2$ (in the period $n=1$ case every ray is in the wake of $W$). Here we use the fact that entropy is continuous along the combinatorial arc from $0$ to $W$ \cite[Theorem~4.9]{Jung}, so there is a postcritically finite parameter $c_3\prec c_1$ with $\tilde h(c_3)\ge \tilde h(c_1)-\eps$. By Proposition~\ref{Prop:BoundEntropyIncrease} there are at most finitely many veins ending at $[c_3,c_1]$ with entropy variation greater than $\eps$. Therefore, $\theta$ is continuous among all rays outside of the wake of $W$.
\end{proof}

\begin{theorem}[Continuity of Entropy Near Veins]
\label{Thm:ContinuityNearVeins} \lineclear
Suppose $\theta\in\Circle$ is such that topological entropy is continuous along the (combinatorial) vein $[0,c(\theta)]$ connecting the parameters $0$ to $c(\theta)$ in $\M$. Then $\tilde h$ is continuous for all parameters on $[0,c(\theta)]$, and $h$ is continuous at all angles $\phi$ that correspond to parameters on $[0,c(\theta)]$.
\end{theorem}

\begin{remark}
It may be helpful to explain the statement. Let $\gamma\colon[0,1]\to \C$ be a parametrization of the (combinatorial) arc $[0,c(\theta)]$. Then the hypothesis says that $\tilde h(\gamma(t))$ is continuous for $t\in[0,1]$ (only considering parameters along the arc). The conclusion is that then $\tilde h\colon\M\to[0,\log 2]$ is continuous at $\gamma(t)$ for all $t$ (where $\gamma(t)$ is now viewed as an element of $\C$, not just of the arc). Note that this hypothesis is known to be true for all angles $\theta\in\Circle$, except when $c(\theta)$ is an endpoint of $\M$ at an irrational angle \cite[Theorem~4.9]{Jung}; we will treat the missing case in Section~\ref{Sec:IrratEndpoints}. 
\end{remark}

\begin{proof}
We start the proof with an auxiliary consideration that does not involve the angle $\theta$. Suppose there are two parameters $c_a\prec c_b\in\M$ with $0\le \tilde h(c_b)-\tilde h(c_a)\le\eps$; we allow $c_b$ to be a (combinatorial)  endpoint of $\M$.

Denote by $\text{wake}(c_a)$ the open wake of $c_a$: this is the set of all parameters in $\M$ that are separated from $0$ by two parameter rays landing at (or accumulating at) $c_a$. If $c_a$ is a Misiurewicz-Thurston-parameter, then we set $\text{wake}(c_a)$ to be the subwake of $c_a$ containing $c_b$. Similarly $\text{wake}(c_b)$ is defined; if $c_b$ is a combinatorial endpoint of the Mandelbrot set, then $\text{wake}(c_b)=\emptyset$. Set $W:=\ovl{\text{wake}(c_a)}\sm\text{wake}(c_b)$. By Proposition~\ref{Prop:BoundEntropyIncrease} there are at most finitely many dyadic veins $[c_i,c'_i]$ with $c_i\in [c_a,c_b]$ such that the entropy variation along $[c_i,c'_i]$ exceeds $\eps$. We may suppose that $(c_i,c'_i]\cap [c_a,c_b]=\emptyset$, possibly by replacing $[c_i,c'_i]$ with the closure of $[c_i,c'_i]\sm[c_a,c_b]$.

We will construct a ``reduced wake'' $W'\subset W$ in which the entropy variation is at most $2\eps$. 

By the Branch Theorem \cite{Orsay}, \cite[Theorem~3.1]{MandelBranch}, the points $c_i$ are either Misiurewicz-Thurston parameters or centers of hyperbolic components. In both cases, we will exclude a subwake at $c_i$ from $W$ where the entropy variation is large.
 
If $c_i$ is a Misiurewicz-Thurston-parameter, let $W_i$ be the subwake of $c_i$ containing $c'_i$ and thus $(c_i,c'_i]$. If $c_i$ is the center of a hyperbolic component, say $H_i$, then let $W_i$ be the subwake of $H_i$ that contains $c'_i$ (the root of this wake is a bifurcation parameter on $\partial H_i$).  In both cases, $W_i$ does not contain $c_b$.

Set $W':=W\sm \bigcup_i \ovl{W_i}$ {(recall that the union is finite)}. The external angles corresponding to rays in $W'$ occupy finitely many intervals, and the maximal entropy of these angles occurs either at an interior point or at an endpoint. In the first case, the maximum is at a dyadic angle by Theorem~\ref{Thm:TiozzoConj} (the Tiozzo Conjecture), and our construction is such that this maximum is at most $\tilde h(c_a)+2\eps$. In the second case, it occurs at an angle corresponding to some $c_i\in[c_a,c_b]$ with entropy $\tilde h(c_i)\le \tilde h(c_b)\le \tilde h(c_a)+\eps$
(if $c_i$ is the center of a hyperbolic component, then the ray at the angle with entropy maximum does not land at $c_i$, but at a parabolic boundary point of the same hyperbolic component, with equal entropy).
Therefore, for all parameter rays $R(\phi)\subset W'$ we have $\tilde h(c_a)\le h(\phi)\le \tilde h(c_a)+2\eps$.

Now we start the actual proof: consider a parameter $c\in[0,c(\theta)]$ with $c=c(\phi)$ for some $\phi\in\Circle$. We claim that $\tilde h$ is continuous at $c$ and that $h$ is continuous at $\phi$. Fix $\eps>0$.

We first consider the case that $c(\theta)$ is a combinatorial endpoint and $c=c(\theta)$. We prove continuity of $\tilde h$ at $c(\theta)$ and of $h$ at $\theta$. By hypothesis, entropy is continuous along the combinatorial vein $[0,c(\theta)]$, so there is a $c_a\in\M$ with $0\prec c_a\prec c(\theta)$ and $0\le \tilde h(c(\theta)) -\tilde h(c_a) \le \eps$. Now using the argument from above we construct a reduced wake $W'$ so that all angles $\theta'\in W'$ satisfy $h(\theta')\in[\tilde h(c_a),\tilde h(c_a)+\eps]$.
These angles form a neighborhood of $\theta$, while $W'$ is a neighborhood of $c$. 
This completes the proof when $c=c(\theta)$ is a combinatorial endpoint.

If $c=c(\theta)$ is not a combinatorial endpoint, then we can extend the combinatorial arc $[0,c(\theta)]$, so from now on it suffices to assume that $c\in(0,c(\theta))$, possibly by replacing $\theta$ by a different angle.

The second case is that $c$ is neither a Misiurewicz-Thurston-parameter nor on the boundary of a hyperbolic component. In this case, we may choose an arc $(c_a,c_b)\ni c$ (i.e., $c_a\prec c\prec c_b$) with $0\le \tilde h(c_b)-\tilde h(c_a)\le 2\eps$ and proceed as above. The assumptions on $c$ mean that $c\not\in\partial W_i$, so $c$ is still an interior point of $W'$, and we conclude that $\tilde h$ is continuous at $c$ and $h$ is continuous at $\phi$. 

If $c(\phi)$ is a Misiurewicz-Thurston-parameter, then there are finitely many branches, and the previous argument works separately for all the individual branches. 

The final case is that $c(\phi)$ is on the boundary of a hyperbolic component --- and that case was handled in Theorem~\ref{Thm:ContEntropyHypComps}.
\end{proof}

\begin{theorem}[Continuity of Entropy]
\label{Thm:Continuity} \lineclear
Core entropy $h\colon\Circle\to[0,\log 2]$ is continuous.
\end{theorem}
\begin{proof}
If $\phi\in\Circle$ is such that $c(\phi)$ is on some combinatorial arc $[0,c(\theta)]$ along which entropy is continuous, then $h$ is continuous at $h$: this is the content of Theorem~\ref{Thm:ContinuityNearVeins}. By work of Tiozzo~\cite{TiozzoThesis} and Jung \cite[Theorem~4.9]{Jung}, this is true for all dyadic angles $\theta$. 

This proves continuity of $h$ at all angles $\phi$ except when $\phi$ corresponds to a combinatorial endpoint of $\M$ at irrational angle, or when $\phi$ corresponds to a boundary point of a hyperbolic component at irrational angle. 

The second case, irrational boundary points of hyperbolic components, has been treated in Theorem~\ref{Thm:ContEntropyHypComps}.

The first case will be taken care of in Corollary~\ref{cor:RadialContin}: there we will prove continuity of entropy along all combinatorial arcs $[0,c(\phi)]$ for all irrational endpoints of $\M$, and then the claim follows as above from Theorem~\ref{Thm:ContinuityNearVeins}. 
\end{proof}

\begin{remark}
Recall from Section~\ref{Sec:Definitions} that continuity of $h\colon\Circle\to[0,\log 2]$ implies continuity of $\tilde h\colon\M\to[0,\log 2]$ (there is a natural continuous projection from $\M$ to the ``abstract Mandelbrot set'', and $h$ is naturally defined on the latter, so $\tilde h$ is the composition of two continuous maps).
\end{remark}

%\newpage

\reminder{Current number of reminders: \arabic{treminder}}

%\newpage

\section{Irrational endpoints}
\label{Sec:IrratEndpoints}

In this section we prove that for every combinatorial endpoint $c$ of $\M$, entropy is continuous along the combinatorial arc $[0,c]$. This is known when $c$ is postcritically finite \cite[Theorem~4.9]{Jung}, but we need it in all cases. This proof provides the missing step in the continuity proof in Theorem~\ref{Thm:Continuity}. {In fact, we prove the result somewhat more generally for endpoints that are non-dyadic, whether or not they are irrational.}

We will approximate the non-dyadic endpoints by dyadic ones, and of course we need uniform estimates for the latter (Proposition~\ref{prop:Uniformity}).

\subsection{Hubbard Trees, Automata, and Renormalization}
\label{Sub:TreesAutomataRenormalization}
Consider a dyadic endpoint $c'$ of $\M$ with external angle ${q}/{2^m}$ and Hubbard tree $H'$. Then the critical value and all further postcritical points are endpoints of $H'$, so vertices of $H'$ are either endpoints or branch points. As in Lemma~\ref{Lem:MarkedPointsSubset_new} we add $\alpha$ and $-\alpha$ to the vertex set of $ H'$ (if necessary).   Easy calculations show that $H'$ has
\begin{itemize}
\item $m+1$ endpoints;
\item at most $m-1$ branch points; 
\item at most $2m+2$ vertices  (including $\alpha$ and $-\alpha$); and 
\item at most $2m+1$ edges. 
 \end{itemize}  

 Let $c$ be the postcritically finite parameter where the vein of $c'$ terminates.  We will assume that $c\not=0$ so that we are in the setting of Lemma~\ref{Lem:MarkedPointsSubset_new}. We denote by $H$ the Hubbard tree of $c$; note that $\beta \not\in H$. 
Denote by $x\in H'$ the dynamic counterpart of $c$ as in Definition~\ref{Def:DynamCounterpart}.

\begin{lemma}
The set $\{ p_{c'}^{\circ k}(x): k\ge 0\}\setminus \{\text{vertices of }H'\}$ contains at most $m$ points; equivalently, $p_{c'}^{\circ m}(x)$ is a branch point of $H'$.
\end{lemma}
\begin{proof}
Let $\tilde c$ be the dyadic parameter with $c'\lhd \tilde c$, so $c'$ is directly subordinate to $\tilde c$. Then $c\prec c'$ and $c\prec \tilde c$.
Denote by $\tilde q/ 2^n$ the external angle of $\tilde c$; note that $n<m$.

We will work in the dynamical plane of $p_{c'}$. 
The rays landing at $x$ separate $R(0), R(\tilde q/2^n)$, and $R(q/2^m)$ (Lemma~\ref{Lem:DirectlySubordinateDynamics}; the point $x$ is denoted $x_*$ there), so $x$ has at least three branches in $H'\cup [x(\tilde q/ 2^n),x]$, where $x(\tilde q/ 2^n)$ denotes the landing point of $R(\tilde q/ 2^n)$.

We claim that $p_{c'}^{\circ n}[x(\tilde q/ 2^n),x]\subset H'$. 
Indeed, {for $k\ge 0$} let $T_k$ be the minimal tree connecting the $\alpha$ fixed point to all dyadic endpoints of generation at most $k$. Since endpoints of $p_{c'}(T_{k-1})$ are endpoints of $T_{k-2}$ as well as the critical value, we have $p_{c'}(T_{k-1})\subset T_{k-2}\cup H'$ and, by induction, indeed $p_{c'}^{\circ n}(T_{n})\subset H'$.

Therefore, $p_{c'}^{\circ n}(x)$ has at least $3$ branches in $H'$ because $p_{c'}^{\circ n}:H' \cup [x(\tilde q/ 2^n),x]\to H'$ is locally injective near $x$. Since $m>n$, also $p_{c'}^{\circ m}(x)$ is a branch point (and also $p_{c'}^{\circ (m-1)}(x)$).
\end{proof}

\goodbreak
  
Let us now refine $H'$ by adding the finite set $\{ p_{c'}^{\circ k}(x): k\ge 0\}$ to its vertex set. The new tree, still called $H'$, has 
\begin{itemize}
\item at most $3m+2$ vertices; and 
\item at most $3m+1$ edges.
 \end{itemize}

\begin{corollary}
\label{cor:EdgesOfH}
 The tree $H$ has 
\begin{itemize}
\item at most $2m+1$ vertices; and 
\item at most $2m$ edges.
 \end{itemize}   
\end{corollary}
\begin{proof}
We have a natural injection of marked points (vertices) in $H$ to marked points in $H'$ (Lemma~\ref{Lem:MarkedPointsSubset_new}) that is compatible with the dynamics and with edges connecting marked points (Lemma~\ref{lem:HintoH'}), and the $m+1$ endpoints of $H'$ are not in the range of this injection (Lemma~\ref{Lem:MarkedPointsSubset_new}). 
\end{proof}

\emph{Overview on the argument.}
The key idea of our proof consists of identifying the dynamics of $p$ on $H$ as an embedded subset of the dynamics of $p'$ on $H'$. Since entropy measures the growth rate of choice of orbits of length $n$, the entropy of $p'$ on $H'$ is no less than the entropy of $p$ on $H$, and we need to give an upper bound on the difference. An orbit in $H'$ that realizes the additional choice is one that leaves the embedded image of $H$ in $H'$, and we show that it starts on a single edge $[x,c']$ at the critical value. We show that this edge maps forward homeomorphically a large number of iterations: so if some orbit uses the additional choice, then it will not have any choice for a long time, and this will give an upper bound on the entropy increase. 

\emph{Automata}.
Here and elsewhere, we find it convenient to express some combinatorial properties in terms of \emph{automata}.  We would like to reassure the reader that we only use the basic notion without results from automata theory and hope it will not be distracting. The concept is simple: given a postcritically finite polynomial $p=p_{c}$ with Hubbard tree $H$, we associate to it an automaton $A$ in a natural way, as follows. The states of $A$ correspond to the edges of $H$. There is an arrow in $A$ from edge $e_1$ to edge $e_2$ whenever $p(e_1)\supset e_2$; the number of arrows from  $e_1$ to $e_2$ equals the number of times $p(e_1)$ covers $e_2$ under $p$ (this is well defined because we have a Markov partition). This number of arrows equals $0$ or $1$, except for the unique edge (if any) that contains the critical point in its interior. 
Similarly, denote by $A'$ the natural automaton associated with $p'=p_{c'}\colon H'\to H'$. 

In Lemma~\ref{lem:HintoH'}, we had identified the dynamics on edges of $H$ as a subset of the dynamics on edges of $H'$. We will find it convenient to express this fact by saying that the automaton $A$ can be considered as a sub-automaton of $A'$; this is done in the following lemma. 

\begin{lemma}[Automata and Edges in Hubbard Trees]
\label{Lem:MappingEdgesHubbardTrees} \lineclear
Let $J : H\to H'$ be the embedding of vertices and edges of the Hubbard tree of $H$ into vertices and edges of the Hubbard tree of $H'$ as in Lemma~\ref{lem:HintoH'}. Then $J$ induces an inclusion of automata $A \hookrightarrow A'$ by mapping a state $e$ of $A$ into the state $J(e)$ of $A'$, {and this inclusion is compatible with the number of arrows between states}.

There are exactly two arrows in $A'$ going from $J(A)$ to $A'\setminus J(A)$, and they connect the same states (they start at the state corresponding to the edge containing the critical point, and they end at the state corresponding to the edge $[c',x')$).
\end{lemma}
\begin{proof}
By Lemma~\ref{lem:HintoH'}, the inclusion $J$ injects the set of states of $A$, which are edges of $H$, into the set of states of $A'$, which are edges of $H'$. Moreover, the number of arrows from $a_1$ to $a_2$ (i.e., the degree of the corresponding map on the edge) is equal to the number of arrows from $J(a_1)$ to $J(a_2)$.

The claim concerning the edges from $J(A)$ into $A'\setminus J(A)$ is a straightforward reformulation of the last claim of Lemma~\ref{lem:HintoH'}.
\end{proof}

In view of this lemma, we may simply write $A\subset A'$; this will help us compare the two automata.

We need to discuss whether the dynamics on our trees is \emph{renormalizable}. A quadratic polynomial $p$ is $n$-renormalizable when there exists a proper subset $K_r$  (the ``little Julia set'') of the filled-in Julia set so that $K_r$ is compact, connected, and full and $p_c^{\circ n}\colon K_r\to K_r$ is a proper map of degree $2$, topologically conjugate to another polynomial in $\M$. It is well known that this means that $p$ belongs to a small embedded copy of the Mandelbrot set within itself (here and elsewhere, we ignore the possibility of ``crossed renormalization''; see \cite{McMullenRenormalization,CrossRenorm}: we use ``renormalizable'' in the meaning of ``simple renormalizable''). 

Since $H'$ is dyadic, its dynamics is never renormalizable, but $H$ may be.  
If the critical point of $H$ is periodic of some period $n$, then the dynamics is trivially $n$-renormalizable, and the Hubbard tree resulting from $n$-renormalization is trivial (a single point). This particular case of renormalizability is not related easily to the dynamics of edges and thus to the automaton $A$. However, we have the following.  

%\newpage

The following result is standard and its proof is omitted; compare McMullen~\cite[Sec.~7]{McMullenRenormalization}.
\begin{lemma}[Immediately Satellite Renormalization]
\label{Lem:RenormMinimalPeriod} \lineclear
If $p_c$ is $n$-renormalizable and $n$ is minimal with this property, and the small Julia sets of the $n$-renormalization are denoted $K_i$, 
then either
\begin{itemize}
\item all $K_i$ contain the $\alpha$ fixed point of $p_c$; in this case $H\subset \bigcup_{i=0}^{n-1}K_i$; or
\item the $K_i$ are pairwise disjoint. 
\end{itemize}
\end{lemma}
\begin{remark}
There are renormalizable Hubbard trees in which some $K_i$ intersect and others do not, but this does not happen when the renormalization period $n$ is chosen minimal. 

\looseness-1
In parameter space, Lemma~\ref{Lem:RenormMinimalPeriod} may be expressed as follows. If $p_c$ is renormalizable,  let $\M'$ be the largest renormalization copy of $\M$ containing $p_c$ (corresponding to the least period of \mbox{renormalization}); the two cases in the lemma depend on whether or not $\M'$ touches the main cardioid of $\M$. If not,  the main component of $\M'$ is primitive. (There are many renormalization copies of $\M$ within $\M$ that are non-primitive and that do not touch the main cardioid of $\M$; these are contained in larger renormalization copies of $\M$, and they describe $n$-renormalizable parameters for which $n$ is not minimal.) We call a parameter \emph{immediately satellite renormalizable} if the first case in Lemma~\ref{Lem:RenormMinimalPeriod}
is realized.
\end{remark}

\begin{lemma}[The Hubbard Tree and Renormalization]
\label{Lem:HubbardTreeRenormalization} \lineclear
For every edge $e$ of $H$, at least one of the following is true:
\begin{itemize}
\item there is a $k\ge 0$ such that $p_c^{\circ k}(e)=H$; or
\item $p_c$ is renormalizable and there is a cycle $K_0,K_1,\dots ,K_{n-1}$ of small filled in Julia sets of $p_c$ such that $e\subset \bigcup_{i=0}^{n-1}K_i$.
\end{itemize}
\end{lemma}
%\newpage
\begin{proof}
We start by assuming that $p_c$ is $n$-renormalizable; let $n$ be minimal with this property. We may assume that all $K_i$ are disjoint: if not, then by Lemma~\ref{Lem:RenormMinimalPeriod} all $K_i$ meet at the $\alpha$ fixed point of $p_c$, and thus $H\subset \bigcup_{i=0}^{n-1}K_i$, so the result is clear.

 We need the following properties.
\begin{itemize}
\item[(A)] 
If $T$ is %a union of edges of $H$ that is connected and periodic (as a set)
a periodic connected subset of $H$ such that $T \neq H$ but $T$ contains at least two points, then $T\subset K_i$ for some $i$. 
\end{itemize}
Indeed, since $T$ is periodic but not a singleton,  the forward orbit of $T$ contains a critical point. Hence around $T$ there is a small filled in Julia set of $p_c$.
\begin{itemize}
\item[(B)] 
If $K$ is one of the $K_i$, or an (iterated) preimage thereof, then $H\sm K$ has at most two connected components.
\end{itemize}
Indeed, there must be some $K_i$ so that $H\sm K_i$ is connected (start with an arbitrary $K_0$, and if $H\sm K_0$ is not connected, choose some $K_1\neq K_0$, and then some $K_2$ in a different component of $H\sm K_1$ than $K_0$; this process must terminate at some $K_i$ for which $H\sm K_i$ is connected).  Further, if $K_j$ does not contain the critical point, then $H\sm K_j$ has no more connected components as $H\sm p_c(K_j)$; and if $K_j$ contains the critical point, then $H\sm K_j$ can have at most twice as many components as $H\sm p_c(K_j)$. Therefore, $H\sm K$ has at most two connected components for any iterated pre-image $K$ of $K_0$.

\begin{itemize}
\item[(C)] 
Every vertex $v\in H$ either has $v\in K_i$ for some $i$ or the forward orbit of $v$ is disjoint from $\bigcup _{i=1}^{n-1}K_i$. 
\end{itemize} 
\looseness-1
To see this, consider a non-periodic iterated pre-image $K'$ of $K_i$. Then $H\setminus K'$ has at most two connected components by (B). Since the critical point is contained in some periodic $K_0$, it follows that $K'$ contains no postcritical point. A vertex in $K'$ must thus be a branch point, and this is possible only if $H\sm K'$ has at least three components, which is not the case.

\looseness-1
Suppose now that $e\not \subset \bigcup_{i=1}^{n-1}K_i$ and let the endpoints of $e$ be $v_1$ and $v_2$. Then for some $m\ge 0$ the vertices $p_c^{\circ nm}(v_1)$ and $p_c^{\circ nm}(v_2)$ are periodic. Since  $e\not \subset \bigcup_{i=1}^{n-1}K_i$, property (C) implies that $p_c^{\circ nm}(v_1)$ and $p_c^{\circ nm}(v_2)$ are not in a single $K_i$. 

Set $e':=p_c^{\circ nm}(e) $ and let $q$ be the least common  period of $p_c^{\circ nm}(v_1)$ and $p_c^{\circ nm}(v_2)$. Then $\bigcup_{k\ge 0}p_c^{\circ qk} (e')$ is a periodic subset of $H$; thus $\bigcup_{k\ge 0}p_c^{\circ qk} (e')=H$ by (A). It is easy to see that $\bigcup_{k\ge 0}p_c^{\circ qk}(e') = p_c^{\circ qm}(e') $ for some $m$ (the set $p_c^{\circ qk}(e')$ is increasing in $k$). 

Finally, if $p_c$ is not renormalizable, choose an edge $e=[v_1,v_2]$ and an $m\ge 0$ so that $p_c^{\circ m}(v_1)$ and $p_c^{\circ m}(v_2)$ are periodic. Let again $e':=[p_c^{\circ m}(v_1),p_c^{\circ m}(v_2)]$; then $\bigcup_{k\ge 1}p_c^{\circ k} (e')$ is periodic, and by (A) it follows that either $p_c$ is renormalizable or $\bigcup_{k\ge 1}p_c^{\circ k} (e')=H$, and in the latter case again $p_c^{\circ m} (e')=H$ for some $m$. 
\end{proof}

%\newpage

We want to relate renormalization of $H$ to properties of the associated automaton $A$. We write 
$A=A_n\cup A_r$ so that 
\begin{itemize}
\item $A_n$ (non-renormalizable edges) contains all states that, for some fixed finite iterate, reach all states of  $A$ simultaneously; and 
\item $A_r$ (renormalizable edges) contains all states from which {not all} of $A$ can be reached simultaneously: these correspond to edges within the Hubbard trees of  ``smalls Julia sets'' corresponding to renormalization domains).
\end{itemize}
We should remark that similar notions can be found in \cite[Sec~3.3]{Jung}.

\begin{lemma}[Renormalization and Automata]
\label{Lem:RenormAutom} \lineclear 
Suppose that $p_c$ is postcritically finite. 
We have $A_r=\emptyset$ if and only if  
\begin{itemize}
\item 
either $p_c$ is non-renormalizable 
\item
or $c$ is periodic of some period $n$ and $p_c$ is $n'$-renormalizable only for $n'=n$ and so that all little Julia sets are disjoint.
\end{itemize}

If $p_c$ is renormalizable, let $n$ be minimal so that $p_c$ is $n$-renormalizable. Let $K_0,K_1,\dots K_{n-1}$ be the cycle of small filled-in Julia sets of $p_c$. Then $A_r$ is the set of states of $A$ so that the associated edges are within $\bigcup_{i=0}^{n-1}K_i$.
\end{lemma}
\begin{proof}
Let again $H$ denote the Hubbard tree of $p_c$.
If $p_c$ is non-renormalizable, then it follows from Lemma~\ref{Lem:HubbardTreeRenormalization} that every edge $e\subset H$ has a $k$ so that $p_c^{\circ k}(e)=H$ for all large $k$, so $A_r=\emptyset$. 

Suppose $p_c$ is renormalizable and let $n$ be minimal so that $p_c$ is $n$-renormalizable. If $e$ is an edge of $H$ such that $e\not\subset \bigcup_{i=1}^{n-1} K_i$, then every sufficiently high iterate of $e$ will cover all of $H$ (Lemma~\ref{Lem:HubbardTreeRenormalization}), and the associated state of $A$ is in $A_n$. The other case is that $e\subset \bigcup_{i=0}^{n-1}K_i$; then by Lemma~\ref{Lem:RenormMinimalPeriod} either the $K_i$ are disjoint or they all touch at the $\alpha$ fixed point (which is a vertex of $H$ by out convention), and in both cases we have $e\subset K_i$ for a unique $i$. In this case, the orbit of $e$ will follow the orbit of $K_i$ and every edge is in only one $K_i$, so it follows that the associated state is in $A_r$.

\looseness-1
We conclude that if $p_c$ is renormalizable, then $A_r=\emptyset$ if and only if $\bigcup_{i=1}^{n-1} K_i$ contains no edge of $H$, and that is the case if and only if every ``little Hubbard tree'' $H_i\subset K_i$ (corresponding to the Hubbard tree after renormalization) is trivial and the $K_i$ are all disjoint. Finally, the little Hubbard trees are trivial if and only if $c$ is periodic of period $n$. 
\end{proof}

An automaton is called \emph{irreducible} if every state of $A$ can be reached from every other. This is certainly the case when $A$ is non-renor\-ma\-li\-zable, but may also happen in the renormalizable case: for instance, the Hubbard tree of the rabbit polynomial with a superattracting $3$-cycle has its Hubbard tree in the form of a topological \textsf{Y} where the three edges are permuted cyclically: the automaton has the form $e_0\to e_1\to e_2\to e_0$ and is irreducible, but the dynamics is renormalizable. The difference is that no edge covers all of $A$ after the same number of iterations (see Lemma~\ref{Lem:RenormAutom}).

If $A_r\neq\emptyset$, so that the dynamics is renormalizable,  we may have $A_n=\emptyset$ or $A_n\neq \emptyset$. The next lemma shows that the former case happens only in the case of immediate satellite renormalization.

\begin{corollary}[Existence of Limit]
\label{Cor:ExistenceLimit} \lineclear
In the dynamics of $p_c$, the limit $\lim_{n\to\infty} (1/n)\log N(n)$ exists whenever $p_c$ is not immediately satellite renormalizable.
\end{corollary}
\begin{proof}
Consider a parameter $c\in\M$ such that $c$ is not immediately satellite renormalizable. Let $e$ be the unique edge in $[\alpha, -\alpha]\subset H$ such that $e$ is attached to $\alpha$. By Lemma~\ref{Lem:RenormAutom} we have $e\in A_n$. Therefore, there is a $k\ge 0$ such that $p_c^{\circ k}(e)=H$, see Lemma~\ref{Lem:HubbardTreeRenormalization}. Thus $e$ contains as many pre-critical points of generation $n+k$ as the number of pre-critical point of generation $n$ in the entire $H$. Since $e\subset [\alpha,-\alpha]$ we get $\sup _{i\le n+k}N_c(i)\le N_c(n)$, hence
\[
\limsup_n\frac 1 n \log N_c(n)=\lim_n\frac 1 n \log N_c(n).
\] 
 If $c$ is not postcritically finite and not an endpoint of $\M$ (that is, $c$ is associated to two external angles in $\M$), then there are two non-immediately-satellite-renormalizable parameters $c_1$ and $c_2$ with $N_{c_1}(n)\le N_c(n)\le N_{c_2}(n)$ and $0\le \tilde h(c_2)-\tilde h(c_1)\le\eps$ for arbitrary $\eps>0$ and the result holds as well (here we use continuity of entropy).

Finally, if $c$ is a non-immediately-satellite-renormalizable endpoint, then by continuity of $\tilde h$, for any $\eps>0$ there exists a postcritically finite non-immediately-satellite-renormalizable parameter $c_1\prec c$ so that $\tilde h(c)-\tilde h(c_1)\le \eps$. 
By monotonicity, we have 
\[
\liminf \frac 1 n N_{c_1}(n)  \le \liminf \frac 1 n N_c(n) \;,
\]
and 
\begin{align*}
\limsup \frac 1 n N_c(n)&=\tilde h(c)\le \tilde h(c_1)+\eps=\liminf \frac 1 n N_{c_1}(n)+\eps
\\ &\le \liminf \frac 1 n  N_{c}(n)+\eps
\;.
\end{align*}
Since $\eps>0$ was arbitrary, the claim follows in this case too.
\end{proof}

\begin{lemma}[Immediate Satellite Renormalization and $A_n=\emptyset$]
\lineclear
A polynomial $p_c$ has $A_n=\emptyset$ if and only if either $p_c$ is immediately satellite renormalizable or it has a superattracting fixed point. \end{lemma}
\begin{proof}

If $p_c$ is immediately satellite $n$-renormalizable, let $K_0,\dots,K_{n-1}$ be the little Julia sets. By Lemma~\ref{Lem:RenormMinimalPeriod}, we have $H\subset \bigcup_i K_i$ and all $K_i$ touch at the $\alpha$ fixed point. Since the $\alpha$ fixed point is a vertex of $H$ every edge is contained in some $K_i$, so the corresponding state is in $A_r$, hence $A_n=\emptyset$. Also, if $p_c$ has a superattracting fixed point, then $H$ is trivial and $A_n=\emptyset$.

For the converse, if $A_n=\emptyset$ but $H$ is not trivial, then $p_c$ must be renormalizable (otherwise $A_r=\emptyset$ by Lemma~\ref{Lem:HubbardTreeRenormalization}), and by Lemma~\ref{Lem:RenormMinimalPeriod} we must have $H\subset\bigcup_i K_i$, so $p_c$ is immediate satellite renormalizable.
\end{proof}

\begin{remark}
Let us note that there are no arrows from $A_r$ to $A_n$, so within $A$ there is no escape from the set of renormalization states $A_r$. However, in $A'\supset A$, if $A_r\neq\emptyset$, then there are two arrows from $A_r$ to $[c',x]$, which is a state in $A'\setminus A$ (see Lemma~\ref{Lem:MappingEdgesHubbardTrees}).
\end{remark}

\goodbreak

We will consider the following special states of $A$ and $A'\supset A$:
\begin{itemize}
\item 
the $0$-state in $A'$ contains the critical point. 
Denoting this state by $e'_0$, we set the $0$-state of $A$ to be $J^{-1}(e'_0)$. This convention is compatible with the inclusion $J\colon A\hookrightarrow A'$ from Lemma~\ref{Lem:MappingEdgesHubbardTrees}.
\item 
the $[c',x]$-state of $A'\setminus A$;
\item 
states of $A$ and of $A'$ that belong to the interval $[\alpha, -\alpha]$.
\end{itemize}

%\newpage

\subsection{Counting Precritical Paths}
\label{Sub:CountingPaths}
A \emph{path} in $A$ or $A'$ is a sequence of arrows so that each arrow starts where the previous arrow ends. The \emph{length} of a path is the number of arrows it contains.  We can also think of a path as a sequence of states so that there is an arrow from every state to the subsequent one (that is, a sequence of edges in the Hubbard tree so that each edge covers the next one under the map). When a path connects two states that are connected by multiple arrows, then there are accordingly multiple paths along this sequence of states (as an example, in $A'$ there are two paths of length $1$ from the $0$-state to the $[c',x]$-state).

We define a \emph{relevant precritical path} in $A'$ or in $A$ as a path that starts at a state in $[\alpha,-\alpha]$ and terminates at the $0$-state.  By basic properties of symbolic dynamics, relevant precritical paths in $A'$ are in bijection with precritical points of $p_{c'}$ in $[\alpha,-\alpha]$ because the critical point of $p_{c'}$ is not a vertex of $H'$. Different relevant precritical paths in $A$ encode different relevant precritical points of $p_c$ in $[\alpha,-\alpha]$.

Every relevant precritical path $s$ in $A'$ has the form  
\begin{equation}
\label{eq:DecompOfs}
s=b_0c_0a_1b_1c_1a_2b_2c_3\dots a_pb_pc_p
\end{equation}
such that (roughly: $a_i, b_i, c_i$ are the sub-paths in $  A'\setminus A$, in $A_n$, and in $A_r$ respectively)
\begin{itemize}
\item 
$a_i$ is an (almost) ``choiceless''  path that starts at the $0$-state, then goes to $[x,c']$, then travels outside of states in $[\alpha, -\alpha]$, and terminates at the first state reached in $[\alpha,-\alpha]\subset A$;
\item 
if $a_i$ terminates at a state in $A_r$, then $b_i=\emptyset$; otherwise, $b_i$ is a path that starts at the state in $[\alpha,-\alpha]$ where $a_i$ terminates and continues while states in $A_n$ are visited (if $i=0$, then instead of the terminal state of $a_i$ we take the initial state of $s$); 
\item 
if $A_r=\emptyset$, then $c_i=\emptyset$; otherwise: if $b_i\not=\emptyset$, then $c_i$ is a path that starts in $A_n$ where $b_i$ terminates, and immediately moves into $A_r$, and otherwise it starts at a state in $A_r$ where $a_i$ terminates. The end of $c_i$ is the $0$-state, and until then the path remains in $A_r$ (again, if $i=0$, then instead of the terminal state of $a_i$ we take the initial state of $s$).
\end{itemize}

Observe that paths in $A'$ that are not in $A$ start on the edge $[x,c']$, so they are described by the $a_i$ that are long and have almost no choice, hence contribute little additional entropy. Indeed, every $a_i$ has length at least $m+1$ because $[x,c']$ needs $m$ iteration to reach $[\beta,\alpha]$, and might need further iterations to land in $[\alpha,-\alpha]$ (Lemma~\ref{Lem:InjectiveDynamicsLastEdge}). Once it lands there, we are either in $A_n$ and we continue with a path $b_i$ as long as we stay in $A_n$, or we are already in $A_r$ and $b_i=\emptyset$, and $c_i$ continues until the next visit of the $0$ state. In particular, if $A_n=\emptyset$, then $b_i=\emptyset$. 
We will refer to $a_i$ as \emph{excursions} (long almost choice-less parts).

Defining $\ell_i,t_i, k_i$ as the lengths of $a_i, b_i, c_i$ respectively,  we say that $s$ has \emph{combinatorial pattern} $P=(t_0,k_0,\ell_1,t_1,k_1,\ell_2,t_2,k_2,\dots,\ell_p,t_p,k_p)$.

\begin{lemma}[Almost Choiceless Paths]
\label{lem:ChoiceLessPaths} \lineclear
For every length $\ell\ge m+1$, there are at most two possible paths in $A'$ that 
\begin{itemize}
\item start at the $0$-state;
\item then immediately go to the $[c',x]$-state;
\item then travel outside $[\alpha,-\alpha]$; 
\item terminate at a given state in $[\alpha,-\alpha]$
\item and have a length $\ell\ge m+1$.
\end{itemize}
\end{lemma}
\begin{proof}
First, the edge of $H'$ associated with the $0$-state covers $[c,x]$ with degree $2$ under $p_{c'}$. Then $[x,c]$ maps injectively for at least $m+1$ iterations until $[x,c]$ starts to partially cover $[\alpha,-\alpha]$. But this is, by definition, when the path under discussion terminates.
\end{proof}

%\newpage

\begin{lemma}[Number of Paths $c_i$] 
\label{lem:c_iPaths} \lineclear
Suppose $c$ is renormalizable. Let $g$ be the period of the biggest small Mandelbrot set containing $c$. Then there are at most $2^{{k}/{g}}$ paths in $A$ that 
\begin{itemize}
\item start at a given state in $A_n$ or in $A_r$;
\item all subsequent states are within $A_r$;
\item terminate at the $0$-state; and
\item have length $k$.
\end{itemize}
\end{lemma}
\begin{proof}
Let $H_0,H_1,\dots H_{g-1}$ be the cycle of small Hubbard trees associated with the the largest renormalizable Hubbard trees (corresponding to the largest small Mandelbrot set containing $c$). Then the degree $p_c^{\circ gt}:H_{0}\to H_{0}$ is at most $2^t$ for all $t$. Therefore, there are at most $2^{t}$ paths in $A_r$ with length $k\in \{ gt,gt+1,\dots, g(t+1)-1\}$ that terminate at the $0$-state; so for given length $k$, the number of such paths is at most $2^t=2^{\lfloor k/g \rfloor}$ (and we have not even counted the first step from a given state of $A$ to $A_r$). 
\end{proof}

%\newpage

Fix a combinatorial pattern $P=(t_0,k_0,\ell_1,t_1,k_1,\ell_2,t_2,k_2,\dots,\ell_p,t_p,k_p)$ and let $n=|P|=t_0+k_0+\sum_{i\ge 1}^p(\ell_i+t_i+k_i)$. 
When comparing entropy in $A'$ and in $A$, we will consider the additional relevant precritical paths in $A'$ and show that they correspond to relevant precritical paths in $A$ of bounded length, so that there are not too many additional paths in $A'$. More precisely, if an excursion has length $\ell_i\ge 3m$ then the new path within $A$ will be shorter (or have equal length) than before. 

We thus introduce a quantity $\kappa$, called \emph{uncertainty of $P$}, that measures the possible increase of length as follows:
\[
\kappa(P):= \frac 1 n \sum_{\ell_i<3m}(3m- \ell_{i})
= \frac 1 n  \sum_i  \max(0,3m-\ell_i)
\]
(the first sum is taken over all $\ell_i$ that are less than $3m$). Higher values of $\kappa$ create problems as the paths in $A$ might be shorter than the corresponding paths in $A'$. Since all $\ell_i>m$ and $\sum\ell_i\le n$, we have $3m-\ell_i\le 2m < 2\ell_i$ and $\kappa(P)\in[0, 2]$.

Denote by $N'(P)$ the number of all precritical paths in $A'$ with pattern $P$.
For $\kappa\in [0,2]$ define 
 \[
N'(\kappa,n):= \sum_{\kappa(P)\le \kappa,\  |P|=n} N'(P)
\;,
\]
the numbers of precritical paths in $A'$ with small uncertainty. 
We define $S'(\kappa,n)$ to be the corresponding set of relevant precritical paths with small uncertainty, so that $N'(\kappa,n)=|S'(\kappa,n)|$. 

%\newpage

\begin{lemma}[Replacing a Path in $A'$ by a Path in $A$]
\label{lem:substitution} \lineclear
Suppose $A_n\neq\emptyset$.
If $s=b_0c_0 a_1b_1c_1a_2b_2c_3\dots a_pb_pc_p$ is a precritical path in $A'$ with uncertainty $\kappa$, then there are paths $a^*_i$ in $A_n$ with lengths in $\{m,\dots,3m\}$ such that 
$s^*:=b_0 a^*_1 b_1a^*_1b_2a^*_2\dots a^*_p b_pc_p$ is a path in $A$. If $n$ is the length of 
$s$, then $s^*$ has length at most $n + \kappa n$.
\end{lemma}
\begin{proof}
Recall that some $b_i$ might be empty paths. Choose any state $a\in A_n$. For  convenience, we say that $a$ is the beginning and the end of every empty $b_i$ with $i<p$. If $b_p$ is empty but $c_p$ is not, then we say that the beginning of $b_p$ is the beginning of $c_p$. If  $b_pc_p$ is empty, then the beginning of $b_p$ is the $0$ state.

Since $A_n$ is irreducible and has less than $2m$ vertices (Corollary~\ref{cor:EdgesOfH}) we may replace every $c_ia_{i+1}$ by a path $a^*_{i+1}$ in $A_n$ of length at most $2m$ so that $a^*_{i+1}$ connects the end of $b_{i}$ with the  beginning of $b_{i+1}$ (which are by definition both in $A_n$); by adding up to $m$ arbitrary steps at the beginning, we may arrange things so that $a^*_i$ has length in $\{m,\dots,3m\}$.

Since the length of each $a_i$ is at least $m$, this procedure increases the length of $s$ by at most $\kappa n$ (and even shortens it whenever $c_ia_{i+1}$ has length greater than $3m$). 
\end{proof}

\begin{lemma}[Counting Patterns]
\label{lem:PatternsGrowth} \lineclear
The quantity  
\[
 \limsup_{n}\frac{1}{n}\log\left(\# \{\text{patterns of length }n\}\rule{0pt}{11pt}\right)
\]
tends to $0$ as $m$ tends to infinity. 
\end{lemma}
\begin{proof}
Every pattern $P$ is uniquely characterized by a non-decreasing sequence of positive integers
\[t_0,t_0+k_0,t_0+k_0+\ell_1, t_0+k_0+\ell_1+t_1, \dots. \]
Since $\ell_i\ge m+1$, for every $q\in\{0,\dots,\lfloor n/m\rfloor\}$ the interval $[qm,(q+1)m)$ contains at most $3$ elements of the above sequence; and the same is true for the final interval $[\lfloor n/m\rfloor m+1, n]$. Therefore, the number of all patterns is bounded by $Z(n,m):=((m+1)^3)^{(n/m)+1}=e^{{3(n+m)\log(m+1)}/{m}}$. For fixed $m$, we have $\limsup_n (1/n)\log Z(n,m)=3\log(m+1)/m$, and indeed this tends to $0$ as $m\to\infty$.
\end{proof}

%\newpage

{Recall that $c'$ is a dyadic parameter of generation $m$ and $c$ is the parameter where the vein of $c'$ terminates. }

\begin{proposition}[Uniformity]
\label{prop:Uniformity}\lineclear
For every $\varepsilon >0$ there are $\overline g, \overline m, \ovl\kappa>0$ depending only on $\varepsilon$ but not on $c$ and $c'$ such that the following holds. If 
\begin{itemize}
\item $m\ge \overline m$, 
\item either $c$ is non-renormalizable, or the period of the largest small Mandelbrot set containing $c$ is at least $\overline g$, 
\item $\kappa\le \ovl \kappa$, and
\item $h$ is the entropy of the parameter $c$,
\end{itemize}
then
%\newpage
\[
N'(\kappa,n)\le C e^{(h+\eps)n}  %  \sum_{\nu\le n}N(\nu)
\]
for some constant $C>0$ depending on $c$.
\end{proposition}
\begin{proof}
Let $N(n)$ count the number of relevant precritical paths in $A$ of generation $n$. Since $h$ is the entropy of $c$, there is a constant $C_1>0$ such that
\[
N(n)\le C_1e^{(h+\eps/3)n} \;.
\]
Suppose first that $A_n\neq\emptyset$.
Our first claim is that there are at most $ 2^{{n}/{g}+{n}/{m}}$ precritical paths
\[
s=b_0c_0 a_1b_1c_1a_2b_2c_3\dots a_pb_pc_p
\]
with fixed $(b_i)_{i\le p}$ of a given pattern $P$ of length $n$. Indeed, the beginning of the $c_i$ is fixed by $b_i$, or by $a_i$ if $b_i$ is empty, and the end is at the $0$ state, so by Lemma~\ref{lem:c_iPaths} there are at most $2^{k_i/g}$ choices for each $c_i$ and in total at most $2^{n/g}$ choices for all $c_i$ combined (in the non-renormalizable case, the $c_i$ are empty and there is no choice at all). Each $a_i$ has at most two choices by Lemma~\ref{lem:ChoiceLessPaths}, and since their length is at least $m$, there are no more than $2^{n/m}$ such choices.  

%\newpage

By Lemma~\ref{lem:substitution} we may substitute $c_ia_i$ by $a^*_i$ with $m\le |a^*_i|\le 3m$ and get a precritical path
\[s^*=b_0a^*_1b_1a^*_2b_2\dots a^*_pb_p c_p\] 
in $A$ with length at most $ n+\kappa(P) n $. Denoting the length of $a_i^*$ by $\ell_i^*$, we call the numbers $P^*=(t_0,0, \ell^*_1, t_1,0,\ell^*_2,\dots,\ell^*_p,t_p,k_p)$ the pattern of $s^*$. 

Our next claim is that for fixed patterns $P$ and $P^*$, the number of triples $(s,(b_i)_{i\le p},s^*)$ is at most
\[
2^{{n}/{g}+ {n}/{m}}C_1 e^{(h+\eps/3)(n+\ovl\kappa n)}=C_1e^{((h+\eps/3)(1+\ovl \kappa)+ 1/g+1/m)n}
\;.
\]
Indeed, the number of paths $s$ for a given pattern $P$ with fixed $b_i$ is at most $2^{{n}/{g}+ {n}/{m}}$. The length of $s^*$ is specified by $P^*$; denote it by $\mu$.  Since each $s^*$ is a precritical path in $A$ and we have we have $\mu\le n+\kappa(P)n$ by Lemma~\ref{lem:substitution},  the number of different $s^*$ is bounded by 
\[
N(\mu)\le  C_1e^{(h +\eps/3)(n+\kappa (P)n)}\le C_1 e^{(h+\eps/3)(n+\ovl\kappa n)} 
\;.
\]
Since the $b_i$ are determined by the paths $s^*$, the claim is proved.

If $g$ and $m$ are sufficiently large and $\ovl \kappa$ is sufficiently small, then $(h+\eps/3)(1+\ovl \kappa)+ 1/g+1/m\le h+2\eps/3$, and the number of triples $(s,(b_i)_{i\le p},s^*)$ {(still for fixed patterns $P$ and $P^*$)} is bounded by
\[
C_1e^{(h+2\eps/3)n}
\;.
\]

Since every $s\in S'(\kappa,n)$ is a part of at least one triple $(s,(b_i)_{i\le p},s^*)$ for some patterns $P$ and $P^*$ with $\kappa(P)\le \kappa$ we get the estimate 
\[
N'(\kappa, n)\le  \left|\{(P,P^*)\}\right| C_1  e^{(h+2\eps/3) n} 
\;,
\]
where $\left|\{(P,P^*)\}\right|$ denotes the number of pairs of patterns $P$ and $P^*$ with $|P|=n$ and $|P^*|\le n+\kappa(P)n$. By Lemma~\ref{lem:PatternsGrowth}, we have $\left|\{(P,P^*)\}\right| \le C_2e^{(\eps/3) n}$ for some constant $C_2>0$ when $m$ is sufficiently large. We get 
\[
N'(\kappa, n)\le C_1C_2 \,  e^{(h+\eps) n} 
\;;
\]
this finishes the proof if $A_n\neq\emptyset$.

The case $A_n=\emptyset$ is simpler: here we have all $b_i=\emptyset$. Therefore, every $s\in A'$ is of the form
\[
s=c_0 a_1c_1a_2c_2\dots a_pc_p
\]
and it is easy to see that there are at most $ 2^{{n}/{g}+{n}/{m}}$ precritical paths in $A'$. (By Lemma~\ref{lem:c_iPaths} there are at most $2^{k_i/g}$ choices for each $c_i$ and in total at most $2^{n/g}$ choices for all $c_i$ combined. Each $a_i$ has at most two choices by Lemma~\ref{lem:ChoiceLessPaths}, and since their length is at least $m$, there are no more than $2^{n/m}$ such choices.)  Therefore, if $g\ge \bar g$ and $m\ge \bar m$ are sufficiently big, then 
\[
N'(\kappa, n)\le C \,  e^{\eps n} 
\;. \qedhere
\]
\end{proof}

%\newpage

\subsection{Continuity at Non-Renormalizable Irrational Endpoints}
Let $c_{\infty}$ be a non-dyadic endpoint of $\M$ that is not renormalizable; the case that $c_\infty$ is renormalizable will be treated in Section~\ref{Sub:Renormalizable}.

There is a sequence of dyadic veins $[c_i,c'_i]$ approximating $c_{\infty}$ in the following way \reminder{should we replace $i$ by $\nu$, because $i$ is in different use?}
\begin{itemize}
\item $c_1\prec c_2\prec\dots \prec c_{\infty}$;  
\item $c_{i+1}\in [c_i,c'_i]$; and
\item $ \dots c_i \lhd \dots \lhd  c'_2\lhd c'_1$.
\end{itemize}
Figure~\ref{Fig:IrrationalEndpoint} illustrates the arrangement of these points.
Note that once $c_1$ is chosen, the remaining parameters are uniquely determined: $c'_i$ is the dyadic of least generation with $c'_i\succ c_i$ {in the same subwake of $c_i$ as $c_\infty$}, and $c_{i+1}$ is the branch point in the vein of $c'_i$ where the vein to $c_\infty$ branches off.

\begin{figure}[htb]
\framebox{\includegraphics{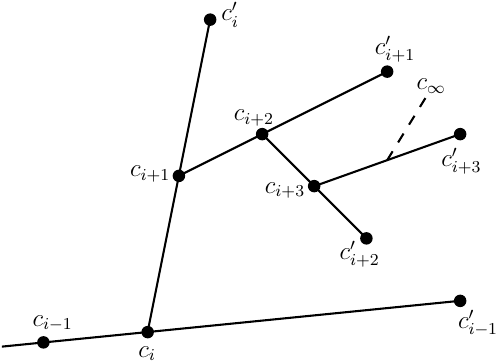}}
\caption{The relative (combinatorial) positions of the parameters $c_i$, $c'_i$, and $c_\infty$  used in the proof.}
\label{Fig:IrrationalEndpoint}
\end{figure}

It may be that individual $c_i$ are renormalizable, but only finitely many of them are $n$-renormalizable for any fixed $n$, so the renormalization periods of $c_i$ (if any) tend to $\infty$; hence the prerequisite of Proposition~\ref{prop:Uniformity} is satisfied for sufficiently large $i$; we will use this in Proposition~\ref{lem:RefOfprop:Uniformity}.

Similar to the previous discussion we specify the following objects:
\begin{itemize}
\item $H'_i$ is the Hubbard tree of $c'_i$ with dynamics $p_{c'_i}:H'_i\to H'_i$ \,;
\item
$S'_{i}(n)\subset [\alpha,-\alpha]\subset H'_i$ is the set of relevant precritical points of generation $n$ in $[\alpha,-\alpha]$  for the parameter $c'_i$\,;
\item 
$N'_{i}(n):= |S'_i(n)|$ is its cardinality;
\item 
$h_i=\tilde h(c_i)$ and $h'_i=\tilde h(c'_i)$ are the entropies;
\item 
$m_i$ is the generation of the dyadic parameter $c'_i$\,; 
 and
\item for $j\le i$ we denote by $x_j^{(i)}\in H_i$ the dynamical counterpart of $c_j$ in $p_{c'_i}:H'_i\to H'_i$ as in Definition~\ref{Def:DynamCounterpart}.  
\end{itemize}

%\newpage

We will do our considerations for a fixed Hubbard tree $H'_i$ and often suppress $i$ from the notation. 
We will now introduce the sets $S'_i(j,\kappa, n)$ for all $j\le i$; these are the sets of relevant precritical points of generation $n$ with uncertainty at most $\kappa$ (in analogy to $S'(\kappa,n)$ in Section~\ref{Sub:CountingPaths}), but subject to a certain relation with respect to the points $x_j^{(i)}$. The special case $j=i$ is exactly the case that was considered in Section~\ref{Sub:CountingPaths}.

Specifically, consider a precritical point $y\in S'_i(n)\subset H'_i$.  Let  \begin{equation}
 \label{eq:Iter:K:ell}0=\tilde \ell_0 <\tilde t_0<\tilde \ell_1<\tilde t_1< \dots < \tilde t_p=n
 \end{equation} be the iteration times of $y$ (depending on $j$) uniquely specified as follows:
\begin{itemize}
\item 
$\tilde t_k>\tilde \ell_k$  is the first time so that $p^{\circ \tilde t_k}_{c'_i}(y) \in [ c'_i, x_{j}^{(i)}]$;
\item 
$\tilde \ell_k>\tilde t_{k-1}$  is the first time so that $p^{\circ \tilde \ell_k}_{c'_i}(y)\in [\alpha,-\alpha]$.
\end{itemize}

%\newpage

We also set $\ell_k:=\tilde \ell_k- \tilde t_{k-1}$ and $t_k:= \tilde t_k-\tilde \ell_k$. Clearly, the sequence \eqref{eq:Iter:K:ell} is uniquely specified by $(t_0,\ell_1, \dots, t_p)$.
 We define the \emph{uncertainty of $y$ with respect to $x_j^{(i)}$} as 
\[
\kappa_j(y):= \frac 1 n \sum_{\ell_k<3m_j}(3m_j- (\ell_k+1))
= \frac 1 n  \sum_k  \max(0,3m_j-(\ell_k+1))
\]
(note that here we start counting at the interval $[c'_i,x_j^{(i)}]$, while in Section~\ref{Sub:CountingPaths} we have $\ell_i=|a_i|$, which starts at the $0$-state before going to $[c',x]$; hence in order to be consistent here we have to use $\ell_k +1$). 

%\newpage

We denote by $S'_i(j,\kappa,n)$ the set of all $y\in S'_i(n)$ such that $\kappa_j(y)<\kappa$.

Let us also define 
\[
I_j(y):= \bigcup_{\ell_k<3m_j} \left\{\tilde t_{k-1}, \tilde t_{k-1}+1,\dots, \tilde \ell_{k} \right\}
= \bigcup_{\ell_k<3m_j} \left[ \tilde t_{k-1},\tilde \ell_k \right]\cap\Z
\;;
\]
this is the set of iteration times without choice (from $[c'_i,x^{(i)}_j]$ to $[\alpha,-\alpha]$).

%\newpage

\begin{lemma}[Blocks of Iteration Times]
\label{lem:I_j} \lineclear
In the Hubbard tree $H'_i$, consider some $y\in S'_i(n)$.
If $m_j > 3 m_{j'}$ for some $j>j'$, then the corresponding sets $I_j(y)$ and $I_{j'}(y)$ are disjoint.

Furthermore, 
\( \displaystyle \kappa_j(y) \le  \frac 2 n |I_j(y)|\) for all $j\ge i$. 
\end{lemma}

\begin{proof}
Let us fist show that for all $j\ge i$, the set $I_j(y)$ is a union of blocks of consecutive numbers so that each block has length $\ell_k+1\in[m_j,\dots, 3m_j]$, and 
\begin{itemize}
\item
 its first number $ \tilde t_{k-1}$ is the unique number $t$ in the block that satisfies $p_{c'_i}^{\circ t}(y)\in [c'_i, \alpha]$;  %x_j^{(i)}]$. 
\item
its last number $\tilde\ell_k$ is the unique number $\ell$ in the block that satisfies 
$p_{c'_i}^{\circ \ell}(y)\in [\alpha,-\alpha]$.
\end{itemize}
We have $\ell_k<3m_j$ by definition of $I_j(y)$ and we need to prove the lower bound $\ell_k\ge m_j -1$.

For every integer $k\ge 0$, let $T_k\subset H'_i$ be the minimal tree connecting the $\alpha$ fixed point to all dyadic endpoints of generation at most $k$. In particular, $[\alpha,-\alpha]\subset T_0$. By~\eqref{eq:T_k-2InT_k-1} we have $p_{c'_i}^{-1}(T_{k-2})\subset T_{k-1}$. 

By construction, the parameters $c'_i$ and $c'_{j-1}$ are in different sublimbs of $c_{j}$. Thus, by Lemma~\ref{Lem:DirectlySubordinateDynamics}, in the tree $H'_i$ the arc $[c'_i,x_{j}^{(i)})$ is disjoint from $T_{m_j-1 }$; and we get $[c'_i,x_{j}^{(i)})\cap p_{c'_i}^{-k} [\alpha,-\alpha]=\emptyset$ for all $k< m_j$. This shows that $\ell_k\ge m_j -1$. 

By definition of $ \tilde t_{k-1}$ we certainly have $p_{c'_i}^{\circ \tilde t_{k-1}}(y)\in[c'_i,x^{(i)}_j]\subset [c'_i,\alpha]$, and since $p_{c'_i}^{-1}([c'_i,\alpha])=[\alpha,-\alpha]$, the block would end before the orbit could enter $[c'_i,\alpha]$
again. The claim about the last number is obvious.

Since any block $I_j(y)$ describes a trajectory from $[c'_i,\alpha]$ to $[\alpha,-\alpha]$, the given properties imply that any two blocks $I_j(y)$ and $I_{j'}(y)$ are either disjoint or identical. If $m_j > 3 m_{j'}$, then any block of $I_j(y)$ has greater length than any block of $I_{j'}(y)$, and consequently $I_{j}(y)$ and $I_{j'}(y)$ are disjoint. 

Finally, $3m_j-(\ell_k+1) \le 2m_j \le 2(\ell_k+1)$, and taking the sum we conclude $\kappa_j(y) \le  \frac 2 n |I_j(y)|$ as claimed.
\end{proof}

%\newpage

\begin{lemma}[Surgery Respects Uncertainty]
\label{lem:MonotIncl} \lineclear
For $j<i$ the injection $B:S'_{i}(n)\to S'_{i-1}(n)$ of Proposition~\ref{Prop:InjectionPrecritical} injects $S'_i(j,\kappa,n)$ into $S'_{i-1}(j,\kappa,n)$. 
\end{lemma}

\begin{proof}
We will show that the injection $B$ respects the sequence \eqref{eq:Iter:K:ell}: the orbits of $y$ and $B(y)$ visit the intervals defining this sequence at the same times. This immediately implies that the uncertainty $\kappa$ is preserved. 

Recall that $x_j^{(i)}$ and $x_j^{(i-1)}$ are the dynamical counterparts of $c_j$ in the dynamical planes of $p_{c'_i}$ and $p_{c'_{i-1}}$ respectively. Consider $y\in S'_i(j,\kappa,n)$. It follows from Proposition~\ref{Prop:InjectionPrecritical} part (B), applied to $c_*=0$ and thus $x_*=\alpha$ and $x'_*=\alpha$, that 
\[
p^{\circ t}_{c'_{i}}(y)\in [\alpha,-\alpha] \;\;\text{ if and only if }\;\; p^{\circ t}_{c'_{i-1}}(B(y))\in [\alpha,-\alpha]
\;.
\] 
This takes care of the first pair of corresponding intervals. For the second pair, it follows from Proposition~\ref{Prop:InjectionPrecritical} part (A) that for all $t\ge 0$
\[
p^{\circ t}_{c'_{i}}(y)\in [c'_{i},-\alpha] \;\;\text{ if and only if }\;\; p^{\circ t}_{c'_{i-1}}(B(y))\in [c'_{i-1},-\alpha]
\;.
\]
 By part (B), applied to $c_*=c_j$, it follows that 
\[
p^{\circ t}_{c'_{i}}(y)\in [x_j^{(i)},-\alpha] \;\;\text{ if and only if }\;\; p^{\circ t}_{c'_{i-1}}(B(y))\in [x_j^{(i-1)},-\alpha]
\;.
\] 
Since $[x_j^{(i)},-\alpha]\subset [c'_{i},-\alpha]$ and $[x_j^{(i-1)},-\alpha]\subset [c'_{i-1},-\alpha]$, we conclude that 
\[
p^{\circ t}_{c'_{i}}(y)\in [c'_{i},x_j^{(i)}]  \;\;\text{ if and only if }\;\; p^{\circ t}_{c'_{i-1}}(B(y))\in [c'_{i-1},x_j^{(i-1)}]
\;.
\]

Therefore, our surgery respects the sequence \eqref{eq:Iter:K:ell}, and thus $\kappa(y)=\kappa(B(y))$ and $B(y)\subset S_{i-1}(j,\kappa,n)$. 
\end{proof}

\begin{lemma}[Small Uncertainty]
\label{lem:SmallKappa} \lineclear
For every $\delta>0$ and for every $i'\ge 0$ there is an $i''>i'$ such that for every $n\ge 0$ we have 
\[
S'_{i''}(n)\subset \bigcup_{j\in \{i',\dots ,i''\}} S'_{i''}(j,\delta,n).
\]
\end{lemma}

\begin{proof}
Choose a subsequence of indices $i'=i_1< i_2< \dots < i_\nu$ so that $3m_{i_r}< m_{i_{r+1}}$ and $\nu> 2/\delta$. We show that $i'':= i_\nu$ satisfies the claim of the lemma. It is sufficient to show that for every $y\in S'_{i_\nu}(n)$ there is a $\nu'\le \nu$ such that $\kappa_{i_{\nu'}}(y)<\delta$.  

By Lemma~\ref{lem:I_j} the sets $I_{i_1}(y), I_{i_2}(y), \dots , I_{i_\nu}(y)$ are pairwise disjoint because $3m_{i_r}< m_{i_{r+1}}$. 
Hence
\[
\frac 1 n (|I_{i_1}(y)|+ |I_{i_2}(y)|) + \dots + |I_{i_\nu}(y)|) \le 1.
\]
Since $\kappa_j(y) \le  \frac 2 n |I_j(y)|$ (Lemma~\ref{lem:I_j} again), we have 
\[
\kappa_{i_1}(y)+\kappa_{i_2}(y)+\dots +\kappa_{i_\nu}(y)\le 2 < \nu\delta
\;;
\]
this implies that $\kappa_{i_{\nu'}}(y)<\delta$ for some $\nu'\le \nu$. 
\end{proof}

%\newpage

The next lemma is a corollary of Proposition~\ref{prop:Uniformity}. 
\begin{lemma}[Bound on Precritical Points]
\label{lem:RefOfprop:Uniformity}\lineclear
For every $\eps>0$ there are $i'\ge 0$ and $\overline \kappa>0$ such that if $ j \ge i'$ and $\kappa\le \overline \kappa$, then 
\[|S'_j(j,\kappa,n)|\le C_j e^{(\eps+h_j) n}\] for all $n>0$ and some constant $C_j>0$.
\end{lemma}
\begin{proof}
For given $\eps>0$ fix $\overline g, \overline m, \ovl\kappa>0 $ as in Proposition~\ref{prop:Uniformity} so that $N'(\kappa,n)\le Ce^{(h+\eps)n}$. Since $c_{\infty}$ is not renormalizable, we may choose $i'$ large enough such that for all $j\ge i'$
\begin{itemize}
\item $m_j\ge \overline m$; and
\item $c_j$ is either non-renormalizable
 or the renormalization period of $c_j$ is at least $\overline g$ (we can make this assumption because $c_\infty$ is non-renormalizable). 
\end{itemize}

We will now apply Proposition~\ref{prop:Uniformity} to the pair $c:=c_j$ and $c':= c'_j$; then $h_j=h$ is the entropy of $c=c_j$. 

 Observe first that $|S'_j(j, \kappa,n)| = |S'( \kappa, n)|= N'(\kappa,n)$ after the substitution. Indeed, every relevant pre-critical point of $p_{c'_j}$ is uniquely characterized by a precritical path in $A'$ (again by a fundamental property of the symbolic dynamics because the critical point is in the interior of the $0$-state of $A'$). This bijection preserves the uncertainties: if $y\in S'_j(j,\kappa,n)$ with sequence~\eqref{eq:Iter:K:ell} is identified with a relevant precritical path $s$ with decomposition~\eqref{eq:DecompOfs}, then $\ell_k +1 =\tilde \ell_k - \tilde t_k +1= |a_k|$. Hence $y$ and $s$ have the same uncertainties, and $S'_j(j, \kappa,n)$ and $S'( \kappa, n)$ are in bijection. 

Now the claim immediately follows from Proposition~\ref{prop:Uniformity}.
\end{proof}

\begin{theorem}[Continuity on Vein to $c_{\infty}$, Non-Renormalizable Case]
\label{thm:ContAtIrrNonRen}
%\lineclear
If $c_\infty$ is a non-renormalizable combinatorial endpoint of $\M$, then core entropy is continuous along the combinatorial arc $[0,c_\infty]$.
\end{theorem}

\begin{proof}
We have parameters $c_i$ and $c'_i$ as introduced at the beginning of the section, and we have
\[
\tilde h(c_i) \le \tilde h(c_{i+1}) \le  \tilde h(c_\infty)\le \tilde h(c'_{i+1})\le  \tilde h(c'_i) 
\] 
for all $i$ by radial monotonicity of core entropy (for the first two inequalities), and by Theorem~\ref{Thm:TiozzoConj} (for the last two). 
Therefore, it suffices to prove the following claim:
\emph{
for every $\eps>0$ there is an $\bar i\ge 1$ such that $h'_{i}-h_{i}\le \eps$ for all $i\ge \bar i$.
}

Choose $\overline\kappa$ and $i'$ as in Lemma~\ref{lem:RefOfprop:Uniformity}.
By Lemma~\ref{lem:SmallKappa} there is an $i''\ge i'$ such that  
\[
S'_{i''}(n)\subset \bigcup_{j\in \{i',\dots ,i''\}} S'_{i''}(j,\overline \kappa,n).
\]
By Lemma~\ref{lem:MonotIncl} there is an injection from $S'_{i''}(j,\overline \kappa,n)$ into $S'_j(j,\overline \kappa,n)$ for all $j\in \{i',\dots ,i''\}$. Therefore,
\[
N'_{i''}(n)=|S'_{i''}(n)|\le  \sum_{j=i'}^{i''} |S'_j(j,\overline \kappa,n)|.
\]
By Lemma~\ref{lem:RefOfprop:Uniformity} we have $|S'_j(j,\overline \kappa,n)|\le C_j e^{(\eps+h_j) n}$. Thus
\[
N'_{i''}(n)\le  \sum_{j=i'}^{i''}C_j e^{(\eps+h_j) n} \le \left( \sum_{j=i'}^{i''}C_j\right) e^{(\eps +h_{i''})n} 
\]because $h_{i''}\ge h_{j}$ by monotonicity.
This proves that $h'_{i''}-h_{i''}\le \eps$, and since the sequence $h'_i-h_i$ is decreasing we have $h'_{i}-h_{i}\le \eps$ for all $i\ge i''$.
\end{proof}

%\newpage

\subsection{The Renormalizable Case}
\label{Sub:Renormalizable}

In this section, we assume that $c_\infty$ is renormalizable. 
In order to formulate our statements, we need to briefly review well known facts on renormalization; compare \cite{Polylike, McMullenRenormalization,MiSelfSim,MiRenorm}. 
If $p_c$ is simple $m$-renormalizable, then there exists a ``little Mandelbrot set'' $\M'\subset\M$ consisting of $m$-renormalizable parameters with $c\in\M'$ and a straightening homeomorphism $\chi\colon\M'\to\M$ so that $p_{\chi(c)}$ in the neighborhood of its filled-in Julia set is hybrid equivalent to $p_c^{\circ m}$ on a neighborhood of the little filled-in Julia set (except possibly at the root point $\chi^{-1}(1/4)$). The little Mandelbrot set has a main center $c_0:=\chi^{-1}(0)$ with a superattracting orbit of period $m$. In this case, we say that ``the parameter $c$ is $c_0$ tuned with $\chi(c)$''. Dynamically, the filled-in Julia set $K_{c}$ equals $K_{c_0}$ in which every Fatou component is replaced by a copy of $K_{\chi(c)}$.

%\newpage

\begin{lemma}[Renormalization and Entropy]
\label{Lem:Renormalization} \lineclear
If $\M'$ is a small copy of $\M$ consisting of $m$-renormalizable parameters, then the straightening map  $\chi:\M'\to \M$ satisfies
\[
\tilde h(c)=\max\left(\tilde h(c_0),\frac 1 m\tilde h(\chi( c))\right)
\;
\]
for all $c\in \M'$.
\end{lemma}

\begin{proof}
For convenience, in this proof we will count the number of relevant pre-critical points of \emph{strict generation} $n$; i.e. relevant pre-critical points of $p_c$, and similar for other polynomials, that are in 
\[p_c^{\circ (-n)}(c)\sm \bigcup _{i=1}^{n-1}p_c^{\circ (-i)}(c).\]
Clearly, this count also gives the entropy of $p_c$. (Indeed, suppose $N_c(n)$ counts relevant pre-critical points of strict generation $n$ while $N'_c(n)$ counts  pre-critical points of ``non-strict'' generation $n$; these numbers are different only when a precritical point of generation $n$, i.e.\ an element of $p_c^{\circ (-n)}(c)$, also is precritical of generation $n'<n$, i.e.\ when the critical point is periodic of period dividing $n-n'$. So if $0$ is periodic, say of minimal period $m$, then $N'_c(n) -N'_c(n-m) \le N_c(n)\le N'_c(n)$. Therefore, $N'_c(n)$ and $N_c(n)$ have the same growth rate.)

Let $c_0$ be the center of $\M'$.  Denote by $N_{c_0}(n)$ and $N_c(n)$ the numbers of relevant pre-critical points on $[\alpha,-\alpha]$ of $p_{c_0}$ and $p_c$  of strict generation $n$. Let us denote by $\tilde N_{\chi(c)}$ the number of pre-critical points on $[\beta,-\beta]$ of $p_{\chi(c)}$ of strict generation $n$.  We will prove the lemma by relating $N_{c_0}(n)$, $\tilde N_{\chi(c)}$, and $N_c(n)$, see~\eqref{eq:Lem:Renorm} below.

First let us note that the growth rate of $\tilde N_{\chi(c)}$ is the entropy of $\chi(c)$. Indeed, by Lemma~\ref{lem:DifferCounts}, the precritical points on $[\alpha,-\alpha]$ and those on $[\alpha,\beta]$ give the same entropies; and on $[\beta,-\beta]$ there can be at most twice as many precritical points of any generation $n$ as on $[\alpha,\beta]$. 

Denote by $F_0$ the unique periodic Fatou component of $p_{c_0}$ containing the critical point. A pre-image $F'$ of $F_0$ under $p_{c_0}^{\circ(n-1)}$ will be called \emph{relevant of strict generation $n$} if the center of $F'$ is on $[\alpha,-\alpha]$ and $F'$ is not a pre-image of $F_0$ under $p_c^{\circ i}$ for any $i<n-1$. Then $N_{c_0}(n)$ counts the number of relevant pre-images of $F_0$ of strict generation $n$.

 For a relevant pre-image $F'$ of $F_0$, say of strict generation $n$, the intersection of $\ovl F'$ with the Hubbard tree of $p_{c_0}$ is equal to $\ovl F'\cap [\alpha,-\alpha]$ because the Hubbard tree of $p_{c_0}$ has no branch points in $F'$ (see Claim (B) in the proof of Lemma~\ref{Lem:HubbardTreeRenormalization}). Therefore,  $p_{c_0}^{\circ(n-1)}$ maps $\overline F'\cap [\alpha,\alpha]$ homeomorphically onto $\overline F_0\cap [\alpha,\alpha']$.

Let $K_{\chi(c)}$ be the filled in Julia set of $p_{\chi(c)}$. Then the filled in Julia set $K_c$ of $p_c$ is obtained from the filled in Julia set $K_{c_0}$ of $p_{c_0}$ by replacing the closure of each Fatou component $F'$ of $K_{c_0}$ with $K_{\chi(c)}$. And, moreover, if $F'$ is a relevant pre-image of $F_0$, then $\overline  F'\cap [\alpha,\alpha']$ is replaced by the arc $[\beta,-\beta]\subset K_{\chi(c)}$. The inserted copies of $ K_{\chi(c)}$ are small (periodic and pre-periodic) filled in Julia sets of $p_c$. A small filled in Julia set $K'$ of $p_c$ will be called \emph{relevant} if the intersection $K'\cap [\alpha,-\alpha]$ has more than two points. Equivalently, $K'$ is obtained by inserting $K_{\chi(c)}$ into a relevant Fatou component of $p_{c_0}$. Let us denote by $K_0$ the small filled in Julia set containing the critical point; i.e. $K_0$ is obtained from $F_0$. By construction, $N_{c_0}(n)$ counts the number of relevant pre-images of $K_0$ of strict generation $n$.

Let us now count relevant pre-critical points of $p_c$. If $x$ is such point, say of strict generation $n\ge 0$, then $x$ is within a relevant pre-image $K'$ of $K_0$. Suppose $K'$ has strict generation $i\ge 0$. Then after $i-1$ iteration $x$ is within $K_0$ and $p_c^{\circ (i-1)}(x)$ is a relevant pre-critical point of strict generation $n+1-i$. Since the dynamics of $p_c^{\circ m}:K_0\to K_0$ is identified with $p_{\chi(c)}:K_{\chi(c)}\to K_{\chi(c)}$, there are $\tilde N_{\chi(c)}\left(\frac{n-i}{m}+1\right)$ relevant pre-periodic points of strict generation $n+1-i$ in $K_0$, where we use the convention $\tilde N_{\chi (c)}(t)=0$ if $t\not\in \N$. We get: 
\begin{equation}
\label{eq:Lem:Renorm}
N_c(n)=\sum_{i=0}^{n}  N_{c_0}(i) \tilde N_{\chi(c)}\left(\frac{n-i}{m}+1\right),
\end{equation}
where $N_{c_0}(i)$ counts pre-images $K'$ of $K_0$ and $\tilde N_{\chi(c)}\left(\frac{n-i}{m}+1\right)$ counts the number of relevant pre-critical points of strict generation $n$ within $K'$. Since the growths of $N_c(n)$, $N_{c_0}(n)$, and $ N_{\chi(c)}(n)$ are the entropies of $p_c$, $p_{c_0}$, and $p_{\chi(c)}$ we get $\tilde h(c)=\max\left(\tilde h(c_0),\frac 1 m\tilde h(\chi( c))\right)$.  
\end{proof}

%\newpage

\begin{corollary}[Continuity Along Vein to $c_{\infty}$, General Case.]
\label{cor:RadialContin}
\lineclear
Suppose that $c_{\infty}$ is a non-dyadic endpoint of the Mandelbrot set. Then the entropy is continuous along the vein $[c_{\infty}, 0]$.
\end{corollary}
\begin{proof}
Theorem~\ref{thm:ContAtIrrNonRen} proves the case when $c_{\infty}$ is non-renormalizable. 

The second case is when $c_{\infty}$ is finitely many times renormalizable. Then is there is a renormalization $\chi \colon\M'\to \M$ such that $c_{\infty}\in \M'$ and $c'_{\infty}:=\chi(c_{\infty})$ is non-renormalizable.  Entropy is continuous along the vein $[0,c'_{\infty}]$ by Theorem~\ref{thm:ContAtIrrNonRen}, and continuous along the image $\chi_*^{-1}([0,c'_{\infty}])=[c_0,c_{\infty}]$ by Lemma~\ref{Lem:Renormalization} (where $c_0$ is the main center of $\M'$), and continuous along $[0,c_0]$ by \cite[Theorem~4.9]{Jung} (or Theorem~\ref{thm:ContAtIrrNonRen}). Therefore, entropy is continuous along $[0,c_{\infty}]$.

The final case is that  there is an infinite sequence $\chi_1:\M_1\to \M$,  $\chi_2:\M_2\to \M$, $\dots$  of renormalizations such that $\M_n \subsetneq \M_{n-1}$ and $c_\infty\in \M_n$ for all $n\ge 0$. Let $c_n$ be the center of $\M_n$; i.e. $c_n=\chi_n ^{-1}(0)$. Then $c_1\prec c_2 \prec \dots \prec c_{\infty}$, and it is sufficient to show that $\lim_n \tilde h(c_n)$ is equal to $\tilde h(c_{\infty})$.

Suppose that $\chi_m\colon\M'_n\to \M$ is an $m_n$-renormalization. Using Lem\-ma~\ref{Lem:Renormalization} we get
\[\tilde h(c_{\infty})=\max\left(\tilde h(c_n),\frac 1 {m_n}\tilde h(\chi( c_{\infty}))\right)\le \tilde h(c_n)+\frac 1 {m_n}h(\chi( c_{\infty})).\] But
$\tilde h(c_n)+\frac 1 {m_n}h(\chi( c_{\infty}))\le \tilde h(c_n)+\frac 1 {m_n}\log 2$ has the same limit as $\tilde h(c_n)$ because $\lim_n m_n = +\infty$. Thus entropy is continuous along $[0,c_{\infty}]$.
\end{proof}

\reminder{Total number of reminders: \arabic{treminder}}

%%%%%%%%%%%%%%%%%%%%%%%%%%%

\newpage

\renewcommand{\top}{{\scriptscriptstyle\mathrm{top}}}
\newcommand{\comb}{{\scriptscriptstyle\mathrm{comb}}}

\appendix
\section{\ \\Core entropy and biaccessibility dimension}
\centerline{by Wolf Jung}
\markleft{D.~DUDKO, D.~SCHLEICHER; APPENDIX BY W.~JUNG}

\bigskip

Here various definitions of core entropy shall be discussed and related to
the biaccessibility dimension. On a compact metric space, the topological
entropy of a continuous map is defined by a growth rate, which is referring
to preimages of covers, or to $\eps$-shadowing sets. See
\cite{BrinStuck, deMelovanStrien} for details. When the underlying space is a
compact interval, a finite tree, or graph, several equivalent characterizations
are due to Misiurewicz and others \cite{alm}. These include the growth rate of
horse shoes, laps (monotonic branches), periodic points,
and preimages of a general point.
For real and complex quadratic polynomials, core entropy was defined by
Tao Li \cite{taoli} and Bill Thurston \cite{TanLeiEntropy, bghkltt}
as the topological entropy of $p_c(z)$ on the Hubbard tree;
this definition applies to the postcritically finite case in
particular, and more generally to finite or infinite compact trees, but it
does not work when $c$ is an endpoint with a dense postcritical orbit. 

In a fairly general situation, the filled Julia set $\mathcal{K}_c$ is
path-connected with empty interior. Then $\mathcal{K}_c$ consists of the
Hubbard tree $T_c$\,, the countable family of its preimages, and an uncountable
family of endpoints. The dynamics on $T_c$ is interesting because this tree is
folded over itself, while the iteration does not return to strict preimages 
of $T_c$\,. On the other hand, with the single exception of $c=-2$
\cite{Zdunik, Smirnov, MeerkampSchleicher}, the endpoints form a set of full
harmonic measure, while the external angles of $T_c$ and its preimages form
a set of Hausdorff dimension $<1$. Since all biaccessible points are contained
in arcs iterated to $T_c$\,, these angles are called biaccessible
(or biaccessing). More precisely, the biaccessibility dimension is defined as
follows:
\begin{itemize}
\item For the lamination generated by an angle $\theta\in \Circle$
\cite{ThurstonLaminations}, consider all
angles of non-trivial leaves. Their Hausdorff dimension is the
\emph{combinatorial biaccessibility dimension} $B_\comb(\theta)$. The same
dimension is obtained from pairs of angles with the same itinerary, or from
pairs not separated by the precritical leaves: the diameter joining $\theta/2$
and $(\theta+1)/2$, and its preimages.
\item For a parameter $c\in\mathcal{M}$, the
\emph{topological biaccessibility dimension} $B_\top(c)$ is the Hausdorff
dimension of those angles for which the dynamic ray is landing together with
another ray.
\end{itemize}

These definitions are related analogously to Lemma~\ref{Lem:defEntropy}:

\begin{lemma}[Combinatorial and topological biaccessibility]
\lineclear
Suppose that $\theta\in \Circle$ and $c\in\partial\M$ belongs to the impression
of the parameter ray with angle $\theta$, or $c\in\M$ is hyperbolic and the
ray lands at the corresponding root. Then $B_\comb(\theta)=B_\top(c)$.
\end{lemma}
\begin{proof}[Sketch of proof \cite{SymDyn, Jung}] 
\looseness-1
In the locally
connected case, and neglecting the countable set of angles at precritical
or precharacteristic points, two dynamic rays are landing together if and only
if they are not separated by a precritical ray-pair. (This separation line
contains additional internal arcs when the interior is not empty.)
When $\mathcal{K}_c$ is not locally connected, exceptional sets of angles are
shown to be negligible in terms of Hausdorff dimension: a Cremer periodic point
might be topologically biaccessible \cite{SchleicherZakeri}, but the relevant
angles have Hausdorff dimension 0 \cite{BullettSentenac}. On the other hand, in
an infinitely renormalizable situation, rays with combinatorially biaccessing
angles may fail to land together, or to land at all. But their Hausdorff
dimension is 0 in the case of the main molecule, and it is less than the
dimension of non-renormalizable biaccessing angles in the primitive maximal
case \cite{Jung}.
\end{proof}

The following relation to entropy is due to Thurston \cite{TanLeiEntropy},
relying on earlier work by Furstenberg \cite{Furstenberg}
and Douady \cite{Douady}.

\begin{proposition}[Dimension and entropy of the tree]\label{Prop:Biaccessibility} \lineclear
Suppose that either
\begin{itemize}
\item
 $\mathcal{K}_c$ is locally connected with empty interior, or 
\item $f_c$ is parabolic, or 
\item $f_c$ is hyperbolic so that the attracting orbit has real multiplier.
\end{itemize}
Using regulated arcs when necessary, define the tree $T_c$ as the path-connected hull of
the critical orbit. If $T_c$ is compact, consider the topological entropy of
$p_c(z)$ on $T_c$\,. Then it is related to
the biaccessibility dimension by $h_\top(T_c)=B_\top(c)\cdot\log2$.
\end{proposition}

\begin{proof}
The proof is found in version 1 of \cite{BruinSchleicher} and in
\cite{TiozzoThesis, Jung}:
since $\theta\mapsto z(\theta)$ is a semi-conjugation with finite fibers, we may
consider the topological entropy of the angle-doubling map on the compact set
of angles of $T_c$ \cite[Thm.~II.7.1]{deMelovanStrien}. And this equals the
Hausdorff dimension \cite[Proposition~III.1]{Furstenberg}, except for the
base 2 instead of $e$ in the logarithm of the growth factor $\lambda$. 
\end{proof}

While the definition of $h_\top(T_c)$ requires a compact tree $T_c$\,,
a general notion was given in Definition~\ref{Def:CoreEntropy} in terms
of precritical ray pairs, and its relation to the biaccessibility dimension
shall be discussed now:

\begin{theorem}[Dimension and entropy in general]
\label{Thm:Biaccessibility} \lineclear
Entropy and biaccessibility dimension are related as follows for
all parameters $c\in\mathcal{M}$ and all angles $\theta\in \Circle$\,:
\begin{equation}
\tilde h(c)=B_\top(c)\cdot\log2
\qquad\mbox{and}\qquad
h(\theta)=B_\comb(\theta)\cdot\log2 \ .
\end{equation}
\end{theorem}

\begin{proof} First, suppose that $c=c(\theta)$ is postcritically finite or
belongs to a dyadic vein. In particular, $\mathcal{K}_c$ is
locally connected and $T_c$ is compact with finitely many endpoints. Then
the growth rate of monotonic branches on $T_c$ is equal to the growth rate of
precritical points on $[\alpha_c\,,-\alpha_c]$, so $h_\top(T_c)=\tilde h(c)$,
and Proposition~\ref{Prop:Biaccessibility} applies.

Second, suppose that $c$ is a non-renormalizable irrational endpoint; 
approximate it with biaccessible parameters $c_n\prec c$. Then monotonicity of
$B_\top(c)$ \cite[Proposition~4.6]{Jung} 
(which in turn is a consequence of monotonicity of characteristic leaves of the lamination, in analogy to Lemma~\ref{Lem:Monotonicity}) and continuity of $\tilde h(c)$ give
\begin{equation}
B_\top(c)\cdot\log2\ge\lim B_\top(c_n)\cdot\log2
=\lim\tilde h(c_n)=\tilde h(c) \ .
\end{equation}
For the opposite estimate, note that the plane is cut into pieces successively
by precritical ray pairs, and the angles of a piece of level $n$ form up to $n$
intervals of total length $2^{-n}$ according to
\cite[Lemma~4.1]{BruinSchleicher}. Recall that $N(n)$ is the number of
precritical points of generation $n$ on $[\alpha_c\,,-\alpha_c]$, and denote
the number of level-$n$ pieces intersecting the arc $[\alpha_c\,,-\alpha_c]$
by $V(n)$. Then $V(n)=1+N(1)+\dots+N(n)$ is growing by the same factor
$\lambda=e^h$ as $N(n)$, and the same holds for $n \cdot V(n)$. The
$b$-dimensional Hausdorff measure of the angles of $[\alpha_c\,,-\alpha_c]$ is
estimated as
\begin{equation}
\mu_b\le\lim n \cdot V(n) \cdot 2^{-bn} \ ,
\end{equation}
which is 0 when $b>\log\lambda/\log2=\tilde h(c)/\log2$. So the Hausdorff
dimension is estimated as $B_\top(c)\le\tilde h(c)/\log2$ as well,
implying equality.

\looseness-1
Finally, both $\tilde h(c)$ and $B_\top(c)$ are constant on the main molecule
and on primitive small Mandelbrot sets: for $B_\top(c)$ see
\cite[Theorem~4.7]{Jung}, and $\tilde h(c)$ is constant on a dense subset and
continuous. Moreover, both scale by the period of
immediate satellite renormalization, see Lemma~\ref{Lem:Renormalization}.
Now all cases are covered by the Yoccoz Theorem \cite{MiRenorm}. 
\end{proof}

Continuity of entropy according to Theorem~\ref{Thm:Continuity} gives:

\begin{corollary}[Continuity of biaccessibility dimension]
\lineclear
The biaccessibility dimension $B_\comb(\theta)$ is continuous on $\Circle$ and
$B_\top(c)$ is continuous on the Mandelbrot set $\mathcal{M}$.
\end{corollary}

Proposition~\ref{Prop:Biaccessibility} and Theorem~\ref{Thm:Biaccessibility}
show that the definition
of entropy in terms of precritical points is a generalization
of the original definition in terms of a compact core:

\begin{corollary}[Extending the definition of core entropy]
\lineclear
We have $h_\top(T_c)=\tilde h(c)$ whenever $T_c$ is defined and compact.
\end{corollary}

%%%%%%%%%%%%%%%%%%%%%%%%%%%

%Total number of reminders: \arabic{treminder}

\vskip-0.8mm

\end{document}